\documentclass[pdflatex,sn-mathphys-num]{sn-jnl}
\usepackage{amssymb}
\usepackage{amsmath}
\usepackage{amsthm}
\usepackage{yhmath}
\usepackage{mathdots}
\usepackage{MnSymbol}
\usepackage{a4wide}
\usepackage{color}

\usepackage{enumitem}
\usepackage{nccmath}

\usepackage{blkarray}
\usepackage{multirow}
\usepackage{float}
\usepackage{booktabs}
\usepackage{array}

\usepackage{caption, subcaption}
\usepackage{graphicx}

\usepackage{natbib}
\usepackage[english]{babel}
\usepackage{hyperref}

\usepackage{algorithm}
\usepackage{algpseudocode}

\newtheorem{theorem}{\bf Theorem}[section]
\newtheorem{definition}[theorem]{\bf Definition}
\newtheorem{corollary}[theorem]{\bf Corollary}
\newtheorem{example}[theorem]{\bf Example}
\newtheorem{remark}[theorem]{\bf Remark}
\newtheorem{lemma}[theorem]{\bf Lemma}

\SetLabelAlign{CenterWithParen}{\hfil(\makebox[1.0em]{#1})\hfil}
\def \R{{\mathbb R}}
\def \C{{\mathbb C}}
\def \Q{{\mathbb H}}
\def \rank{\mathrm{rank}}

\def \i{\textit{\textbf{i}}}
\def \j{\textit{\textbf{j}}}
\def \k{\textit{\textbf{k}}}
\newcommand\norm[1]{\left\lVert#1\right\rVert}
\newcommand{\Rn}[1]{%
	\textup{\lowercase\expandafter{\romannumeral#1}}%
}

\def \diag{\mathrm{diag}}

\usepackage{multicol}

\newcommand\x{\times}

\usepackage{tikz,xcolor}
\definecolor{lime}{HTML}{A6CE39}
\definecolor{lightblue}{rgb}{0.0, 0.0, 0.5}
\DeclareRobustCommand{\orcidicon}{%
	\begin{tikzpicture}
	\draw[lime, fill=lime] (0,0)
	circle [radius=0.16]
	node[white] {{\fontfamily{qag}\selectfont \tiny ID}};
	\draw[white, fill=white] (-0.0625,0.095)
	circle [radius=0.007];
	\end{tikzpicture}
	\hspace{-2mm}
}
\foreach \x in {A, ..., Z}{%
	\expandafter\xdef\csname orcid\x\endcsname{\noexpand\href{https://orcid.org/\csname orcidauthor\x\endcsname}{\noexpand\orcidicon}}
}

\begin{document}

\title[
Structure Preserving Algorithms for Quaternion Outer Inverses with  Applications
]{
Structure Preserving Algorithms for Quaternion Outer Inverses with  Applications
}

\author{\fnm{Neha} \sur{Bhadala$^1$\orcidC}}\email{nehabhadala@iisc.ac.in}

\author*[]{\fnm{Ratikanta} \sur{Behera$^2$\orcidA}}\email{ratikanta@iisc.ac.in}

\affil[]{\orgdiv{Department of Computational and Data Sciences}, \orgname{Indian Institute of Science}, \orgaddress{ \state{Bengaluru}, \country{India}}
}


\vspace{.3cm}





\abstract{This study investigates the theoretical and computational aspects of quaternion generalized inverses, focusing on outer inverses and $\{1,2\}$-inverses with prescribed range and/or null space constraints. In view of the non-commutative nature of quaternions, a detailed characterization of the left and right range and null spaces of quaternion matrices is presented. Explicit representations for these inverses are derived, including full rank decomposition-based formulations. We design two efficient algorithms: one leveraging the Quaternion Toolbox for MATLAB (QTFM), and the other employing a complex structure preserving approach based on the complex representation of quaternion matrices. With suitable choices of subspace constraints, these outer inverses unify and generalize several classical inverses, including the Moore–Penrose inverse, the group inverse, and the Drazin inverse. The proposed methods are validated through numerical examples and applied to two real-world tasks: quaternion-based color image deblurring, which preserves inter-channel correlations, and robust filtering of chaotic 3D signals, demonstrating their effectiveness in high-dimensional settings.
}

\maketitle

\section{Introduction}
Quaternions, first introduced by Hamilton \cite{Hamilton2009} in 1843, have proven to be a versatile mathematical tool with applications spanning numerous fields, including quantum mechanics \cite{MR1333599, MR2672082, MR4682073}, signal processing \cite{MR3308953, MR4572388}, color image processing \cite{ChenJiaPeng2021, MR4867322, MR3979957, MR4861347}, and computer graphics \cite{Vince2021}. Quaternions form a non-commutative division algebra. In particular, quaternion matrices have gained significant attention owing to their ability to efficiently represent and process multidimensional data. Various aspects of quaternion matrices, including their algebraic properties and computational methods, were explored in \cite{MR4421890, MR3241695, MR1421264}.

Recent advances in quaternion matrix computation reflect their expanding roles in both theory and practice. For example, decomposition methods such as full rank factorization using Gauss transformations \cite{MR4509094} and structure preserving singular value decomposition algorithms \cite{MR4685298} have enhanced numerical stability in quaternion-based computations. Investigations into LU decomposition strategies \cite{MR3735827} and their error analysis with partial pivoting \cite{MR4685298} further underscore the maturation of quaternion linear algebra. These advancements highlight the growing importance of quaternion matrices in both the theoretical and applied contexts.

Among the various application domains where quaternion matrices play a crucial role, color image restoration has emerged as a particularly significant and active research area. A standard RGB image consists of three highly correlated channels (red, green, and blue), whose coupling plays a crucial role in restoration tasks. Classical monochromatic approaches process each channel independently, which often fails to capture inter-channel correlations and may lead to color distortions. Concatenation-based models partially alleviate this issue, but still lack an explicit algebraic mechanism to encode channel coupling. Quaternion algebra provides a natural and mathematically principled framework for addressing this limitation. By encoding the RGB components into the three imaginary parts of a quaternion, color images can be represented and processed holistically as quaternion matrices, thereby preserving channel correlations under a unified algebraic structure  \cite{MR4471046, MR4039278, MR4947673, MR3979957, MR3322229}. In such settings, the associated mathematical models naturally lead to quaternion linear system. Solving such systems often requires the use of generalized inverses, which extend the concept of matrix inverses to cases where standard inverses do not exist. Generalized inverses provide a robust framework for addressing linear systems and matrix equations, particularly in constrained or ill-posed scenarios. 
This, in turn, motivates the development of efficient numerical linear
algebra methods for quaternion matrices, particularly generalized inverses with prescribed range and/or null space constraints.

The concept of generalized inverses is characterized by certain algebraic conditions \cite{MR69793} that define the specific relationships between a matrix and its inverse. These conditions are:
\begin{equation*}
    (1)~ AXA = A, ~ (2)~ XAX = X, ~ (3)~ (AX)^* = AX, ~ (4)~ (XA)^* = XA.
\end{equation*}
A matrix \( X \in \Q^{n \times m} \) is called a \(\{1\}\)-inverse of \( A \in \Q^{m \times n} \) if it satisfies condition~$(1)$, i.e., \( AXA = A \). Such an inverse is denoted by \( A^{(1)}\). Similarly, matrices fulfilling condition $(2)$ are called $\{2\}$-inverses or outer inverses, denoted by $A^{(2)}$. When conditions $(1)$ and $(2)$ are satisfied, the matrix becomes a $\{1,2\}$-inverse, represented as $A^{(1,2)}$. Among these, the Moore-Penrose inverse $A^{\dagger}$ stands out as it uniquely satisfies all four conditions. Another important type of generalized inverse is the Drazin inverse. For a square matrix \( A \), the index of \( A \), denoted by \( \text{Ind}(A) = k \), is the smallest integer \( k \geq 0 \) such that \( \rank(A^{k+1}) = \rank(A^k) \). The Drazin inverse of \( A \in \Q^{m \times m}\), denoted by \( A^D \), is a unique matrix \( X \in \Q^{m \times m} \) that satisfies the following properties:
\begin{equation*}
    A^{k+1} X  = A^k, ~ X A X = X, ~ A X = X A.
\end{equation*}
In the special case where \( k = 1 \), matrix \( X \) is called the group inverse of \( A \) and is denoted by \( A^{\#} \).

Generalized inverses, particularly the outer and \(\{1,2\}\)-inverses with a predefined range and/or null space, are invaluable for solving constrained linear systems and optimization problems \cite{MR1987382}. These inverses ensure that specific subspace properties are maintained, making them essential tools for various applications. 
An important aspect of outer inverses constrained by the prescribed subspace is their ability to unify several well-known generalized inverses under a common framework. In particular, with suitable choices of subspace constraints, they encompass fundamental cases such as the Moore–Penrose inverse, group inverse, and Drazin inverse. This perspective provides deeper insights into the structural properties of these inverses and their applications in numerical analysis and applied mathematics.

The study of outer inverses for complex matrices with prescribed range and null space constraints has been extensively explored. This inverse is defined by the following conditions:  
\[
XAX = X, ~ \mathcal{R}(X) = \mathcal{R}(S), ~ \mathcal{N}(X) = \mathcal{N}(T),
\]  
and is denoted as \( X = A_{S,T}^{(2)} \). One of the earliest formulations of these inverses was introduced by Urquhart \cite{MR227186}, which laid the foundation for subsequent developments in this area. These representations were further refined and extended in Theorem~$1.3.7$ of \cite{MR3793648} and Theorem~$13$ of \cite{MR1987382}, providing a comprehensive framework for generalized inverses with prescribed subspace constraints. In \cite{MR2368224}, the authors utilized the full rank decomposition of the matrix \( R = ST \) to derive the representation of \( A^{(2)}_{\mathcal{R}(R),\mathcal{N}(R)} \). Several iterative methods have been proposed for computing outer inverses with prescribed range and null space constraints \cite{MR2451534, MR3162349, MR1645022, MR1937251}.  

While the theory of generalized inverses for complex matrices is well-developed, its extension to quaternion matrices remains largely unexplored. To the best of our knowledge, there is limited literature on the generalized inverses of quaternion matrices. Notably, earlier work \cite{MR2802523, MR2775924} established determinantal representations of the Moore-Penrose inverse and outer inverses with prescribed subspace constraints over the quaternion skew field. However, several fundamental aspects remain unaddressed, particularly in terms of computational techniques. Additionally, because of the non-commutative nature of quaternions, both the left and right generalized inverses must be investigated separately. This gives rise to several open questions, some of which are listed below:
\begin{itemize}[noitemsep]
    \item[$(a)$] Can the Urquhart representation for generalized inverses be extended to quaternion matrices?  
    \item[$(b)$] Is it possible to establish full rank decomposition-based representations for these generalized inverses in a quaternion setting?  
\end{itemize}  
These unresolved issues highlight the need for further research in this field. This work aims to advance the theoretical and computational aspects of generalized inverses for quaternion matrices, with a particular focus on outer inverses and \(\{1,2\}\)-inverses with a prescribed range and/or null spaces. The main contributions of this study are summarized as follows:  
\begin{enumerate}[noitemsep]
   \item A few characterizations and representations of the left and right range and null spaces of the quaternion matrices are discussed. 
    \item Both Urquhart-type and full rank decomposition-based representations are established for the outer inverses and $\{1,2\}$-inverses of quaternion matrices with prescribed range and/or null spaces.
    \item Two efficient algorithms are proposed for computing these generalized inverses: a direct algorithm utilizing the Quaternion Toolbox for MATLAB (QTFM), and a complex structure preserving algorithm based on the complex representation of quaternion matrices.
    \item A unified theoretical framework is established, demonstrating that outer inverses with prescribed subspace constraints naturally encompass several classical generalized inverses such as the Moore–Penrose inverse, group inverse, and Drazin inverse as particular instances.
   \item The proposed algorithms are validated using numerical examples, and their practical relevance is demonstrated via two applications: a color image deblurring task and a three-dimensional signal filtering problem.   
   \end{enumerate}  
A key advantage of the proposed representations is their flexibility: an algorithm designed for one type of generalized inverse can be readily adapted to compute the others. This adaptability enables the selection of the most appropriate inverse depending on the needs of the application. Collectively, these contributions enhance both the theoretical foundation and practical computation of generalized inverses in quaternion algebra, with demonstrated relevance to real-world problems such as image and signal processing.

The remainder of this paper is organized as follows. Section~\ref{sec2} introduces the notation and preliminary results required for subsequent discussion. In Section~\ref{sec3}, we explore the outer inverses and \(\{1,2\}\)-inverses of quaternion matrices with prescribed subspace constraints and derive explicit representations.  Section~\ref{sec4} presents efficient algorithms for computing these generalized inverses, supported by numerical examples. In Section~\ref{sec7}, we present two application scenarios to illustrate the practical utility of the proposed quaternion-based framework.

\section{Notation and preliminaries}\label{sec2}
\subsection{Notation}
The notation used in this paper is defined as follows: the sets $\Q^{m \times n}$, $\C^{m \times n}$, and $\R^{m \times n}$ refer to collections of quaternion, complex, and real matrices of size $m \times n$, respectively. In addition, we use $\Q^{m \times n}_{\nu}$ to represent the set of $m \times n$ quaternion matrices with rank $\nu$. The Frobenius norm of a matrix $A$ is denoted by $\| A \|_F$. For a matrix $A \in \Q^{m \times n}$, the notation \( A^T \) represents its transpose, while \( \bar{A} \) denotes its quaternionic conjugate. The conjugate transpose is denoted as \( A^* \). 
If \( A \) is invertible, its inverse is expressed as \( A^{-1} \). Furthermore, in MATLAB, function \text{rand}(m,n) generates an \( m \times n \) matrix with elements sampled from a uniform distribution. For clarity, these notations are used consistently throughout this study.

\section{Notation and preliminaries}\label{sec2}
\subsection{Notation}
The notation used in this article is defined as follows: sets $\Q^{m \times n}$, $\C^{m \times n}$, and $\R^{m \times n}$ refer to collections of quaternion, complex, and real matrices of size $m \times n$, respectively. In addition, we use $\Q^{m \times n}_{\nu}$ to represent the set of $m \times n$ quaternion matrices with rank $\nu$. The Frobenius norm of a matrix $A$ is denoted by $\| A \|_F$. For a matrix $A \in \Q^{m \times n}$, the notation \( A^T \) represents its transpose. The conjugate transpose is denoted as \( A^* \). 
If \( A \) is invertible, its inverse is expressed as \( A^{-1} \). Furthermore, in MATLAB, function \text{rand}(m,n) generates an \( m \times n \) matrix with elements sampled from a uniform distribution. 

\subsection{Preliminaries}
A quaternion is defined as 
$\zeta = \zeta_0 + \zeta_1\i + \zeta_2\j + \zeta_3\k$,
where $\zeta_i \in \R$ for $i = 0,1,2,3$. 
In addition, the quaternion multiplication rules are given by the following defining relations:
$\i^2 = \j^2 = \k^2 = -1, ~  \i\j = -\j\i = \k, ~ \j\k = -\k\j = \i, ~ \k\i = -\i\k = \j$.
Following the above multiplication rules, $\zeta$  can be rewritten as 
$\zeta = (\zeta_0 + \zeta_1\i) + (\zeta_2 + \zeta_3\i)\j = \gamma_1 + \gamma_2\j$,
where $\gamma_i \in \C$ for $i = 1,2$.      
Given a quaternion $\zeta$, its real component is given by \( \Re(\zeta) = \zeta_0 \), and its imaginary part is expressed as \( \Im(\zeta) = \zeta_1 \i + \zeta_2 \j + \zeta_3 \k \). The conjugate of \( \zeta \) is defined as $\bar{\zeta} = \zeta_0 - \zeta_1 \i - \zeta_2 \j - \zeta_3 \k$, and its norm is given by $|\zeta| = \sqrt{\zeta_0^2 + \zeta_1^2 + \zeta_2^2 + \zeta_3^2}$. For a quaternion matrix $Q = (q_{ij}) \in \Q^{m \times n}$, the Frobenius norm is defined as \cite{MR4685298}: 
\begin{equation}\label{eq2.1}
    \|Q\|_F = \sqrt{\sum_{i=1}^{m} \sum_{j=1}^{n} |q_{ij}|^2}.
\end{equation}
A square matrix \( P \in \Q^{n \times n} \) is termed invertible (or nonsingular) if there exists another matrix \( Q \in \Q^{n \times n} \) that satisfies condition $PQ = QP = I.$ 

Next, we recall the left and right linear independence of quaternion vectors over $\Q$.

\begin{definition}\label{def2.1}{\rm\cite{wei2018quaternion}}
    A set of vectors $u_1, u_2, \dots, u_r \in \Q^n$ is said to be left linearly independent if the equation
    $ \alpha_1 u_1 + \alpha_2 u_2 + \dots + \alpha_r u_r = 0$ with scalars $\alpha_1, \alpha_2, \dots, \alpha_r \in \Q$,
    holds only when $\alpha_1 = \alpha_2 = \dots = \alpha_r = 0$. Otherwise, if there exists at least one nonzero scalar \( \alpha_i \) such that the above equation holds, the vectors are called left linearly dependent.

    Similarly, the vectors $u_1, u_2, \dots, u_r \in \Q^n$ are said to be right linearly independent if the equation
    $ u_1 \alpha_1 + u_2 \alpha_2 + \dots + u_r \alpha_r = 0$ with scalars $\alpha_1, \alpha_2, \dots, \alpha_r \in \Q$,
    holds only when $\alpha_1 = \alpha_2 = \dots = \alpha_r = 0$. Otherwise, if there exists a nonzero scalar \( \alpha_i \) such that the above equation holds, the vectors are called right linearly dependent.
\end{definition}
The notions of the right- and left-range and null spaces are defined as follows. 
\begin{definition}\label{def2.2}{\rm\cite{MR2775924}}
Let \( A \in \Q^{m \times n} \). The left range and null spaces of \( A \), denoted by \( \mathcal{R}_l(A) \) and \( \mathcal{N}_l(A) \), are defined as
\begin{equation*}
\mathcal{R}_l(A) = \{ y \in \Q^{1 \times n} \mid y = xA, \; x \in \Q^{1 \times m} \}, ~
\mathcal{N}_l(A) = \{ x \in \Q^{1 \times m} \mid xA = 0 \}.
\end{equation*}
Similarly, the right range and null spaces, denoted by \( \mathcal{R}_r(A) \) and \( \mathcal{N}_r(A) \), are defined as
\begin{equation*}
\mathcal{R}_r(A) = \{ y \in \Q^{m \times 1} \mid y = Ax, \; x \in \Q^{n \times 1} \}, ~
\mathcal{N}_r(A) = \{ x \in \Q^{n \times 1} \mid Ax = 0 \}.
\end{equation*}
\end{definition}

\begin{remark}\label{rem2.3}
    For $A \in \Q^{m \times n}$, the following key properties hold:
    \begin{itemize}[noitemsep]
        \item[$(a)$] $\Q^{1 \times n}$ forms a left quaternion vector space under addition and left multiplication by scalars.
        \item[$(b)$] $\Q^{n \times 1}$ constitutes a right quaternion vector space under addition and right multiplication by scalars.
         \item[$(c)$] Both $\mathcal{R}_l(A)$ and $\mathcal{N}_l(A)$ are finite-dimensional subspaces within $\Q^{1 \times n}$ and $\Q^{1 \times m}$, respectively.
        \item[$(d)$] Similarly, $\mathcal{R}_r(A)$ and $\mathcal{N}_r(A)$ are finite-dimensional subspaces in $\Q^{m \times 1}$ and $\Q^{n \times 1}$, respectively.
    \end{itemize}
\end{remark}

\begin{remark}\label{rem2.4} {\rm\cite{wei2018quaternion}}
    The rank of a quaternion matrix $A$, denoted by $\rank(A)$, is defined as the maximal count of either the right linearly independent column vectors of $A$, or the left linearly independent row vectors of $A$. This equivalence is formalized as
$\rank(A) = \dim(\mathcal{R}_l(A)) = \dim(\mathcal{R}_r(A))$.
    \end{remark}

Inner products play a fundamental role in the quaternion vector spaces. Unlike real or complex vector spaces, the definition of an inner product in a quaternion space depends on whether the space is structured as a left or right vector space. We formally define the inner products in the quaternion vector space as follows.

\begin{definition}\label{def2.5}{\rm\cite{MR2775924}}
Let $\Q^{1 \times n}$ and $\Q^{n \times 1}$ denote the left and right quaternion vector spaces, respectively. The inner products in these spaces are defined as follows:
\begin{itemize}[noitemsep]
    \item[$(a)$] For the left vector space $\Q^{1\times n}$, given row vectors $x$ and $y$, their inner product is
     \begin{equation}\label{eq2.2}
   \langle x,y \rangle_l = xy^*  = \sum_{i=1}^n x_i \bar{y}_i.
   \end{equation}
    \item[$(b)$] For the right vector space $\Q^{n\times 1}$, given column vectors $x$ and $y$, their inner product is
   \begin{equation}\label{eq2.3}
    \langle x,y \rangle_r = y^*x  = \sum_{j=1}^n \bar{y}_j x_j.
    \end{equation}
\end{itemize}
\end{definition}

We now recall the concepts of the left- and right-orthogonal complements.

\begin{definition}\label{def2.6}{\rm\cite{MR2775924}}
Consider a subset $\mathcal{V}$ of the left quaternion vector space $\Q^{1\times n}$ equipped with the inner product defined in \eqref{eq2.2}. Its left orthogonal complement $\mathcal{V}_l^\perp$ is given by
    \begin{equation*}
        \mathcal{V}_l^{\perp} = \left\{ x \in \Q^{1 \times n} \mid \langle x, y \rangle_l = 0 \text{ for all } y \in \mathcal{V} \right\}.
    \end{equation*} 
    Similarly, for a subset $\mathcal{U}$ of the right quaternion vector space $\Q^{n\times 1}$ with the inner product defined in \eqref{eq2.3}, its right orthogonal complement $\mathcal{U}_r^\perp$ is given by
    \begin{equation*}
        \mathcal{U}_r^{\perp} = \left\{ x \in \Q^{n \times 1} \mid \langle x, y \rangle_r = 0 \text{ for all } y \in \mathcal{U} \right\}.
    \end{equation*}
\end{definition}
Next, we discuss the concept of complex representation of a quaternion matrix. For any $Q = Q_1 + Q_2 \j \in \Q^{m \times n}$, where $Q_1, Q_2 \in \C^{m \times n}$, its complex representation is defined as
\begin{equation}\label{eq2.4}
    Q^C := \begin{bmatrix} Q_1 & Q_2 \\ -\bar{Q}_2 & \bar{Q}_1 \end{bmatrix}.
\end{equation}
Let $Q^C_r$ denote the first block row of the matrix $Q^C$, given by $Q^C_r = \begin{bmatrix}
    Q_1 & Q_2
\end{bmatrix}$.
The following lemma presents some fundamental properties of the complex representation matrix that will be useful in subsequent results.
\begin{lemma}{\rm \cite{MR4509094}}\label{pro2.11}
    Let $\alpha \in \R$, $P, Q \in \Q^{m \times n}$, and $R \in \Q^{n \times p}$. Then
    \begin{enumerate}[noitemsep]
        \item[$(a)$] $(\alpha P)^C = \alpha P^C$, $(P + Q)^C = P^C + Q^C$, and $(PR)^C = P^C R^C$.
        \item[$(b)$] $(\alpha P)^C_r = \alpha P^C_r$, $(P + Q)_r^C = P^C_r + Q^C_r$, and $(PR)^C_r = P^C_r R^C$.
        \item[$(c)$] $\left(P^*\right)^C = \left(P^C\right)^*$ and $(P^{-1})^C = (P^C)^{-1}$, if $P^{-1}$ exists.
        \end{enumerate}
\end{lemma}
The following lemma establishes the key relationship between the rank of a quaternion matrix and the rank of its complex representation.
\begin{lemma}{\rm \cite{MR4509094}} \label{lem2.12}
    For any $P \in \Q^{m \times n}$ and $Q \in \Q^{n \times p}$, we have
\[
\rank(P) = \tfrac{1}{2}\rank(P^C),  ~ 
\rank(PQ) \leq \min\{\rank(P), \rank(Q)\}.
\]
\end{lemma}

The following lemmas use the complex representation to establish the existence
of left- and right inverses for full-rank quaternion matrices.
\begin{lemma}\label{rr2}
Let
$J_k=\begin{bmatrix}0&I_k\\ -I_k&0\end{bmatrix}$.
A matrix $C\in\C^{2m\times 2n}$ satisfies $J_m C=\bar{C}J_n$ if and only if there exists a unique $B\in\Q^{m\times n}$ such that $C=B^C$.
\end{lemma}
\begin{proof}
The result follows directly from the definition of the complex representation
of quaternion matrices.
\end{proof}

\begin{lemma}\label{rr3}
Let $A\in\mathbb{H}^{m\times n}$ with $\rank(A)=n$ and $m\ge n$.
Then there exists a matrix $C\in\C^{2n\times 2m}$ such that
$C A^C=I_{2n}$
and $J_n C=\bar{C}J_m$,
where $J_k$ is defined as in Lemma~\ref{rr2}.
\end{lemma}
\begin{proof}
By Lemma~\ref{lem2.12}, $A^C \in \C^{2m \times 2n}$ has full column rank $2n$.
Hence, there exists $C_0\in\C^{2n\times 2m}$ such that
$C_0 A^C=I_{2n}$.
Define
$C=\frac{1}{2}\left(C_0+J_n^{-1}\bar{C_0}J_m\right)$.
A direct computation using the identity
$J_m A^C=\bar{(A^C)}J_n$
shows that $C A^C=I_{2n}$.
Moreover, $C$ satisfies $J_n C=\bar{C}J_m$ by construction.
\end{proof}

\begin{lemma}\label{rr4}
Let $A\in\Q^{m\times n}$ with $\rank(A)=m$ and $n\ge m$.
Then there exists a matrix $D\in\C^{2n\times 2m}$ such that
$A^C D=I_{2m}$ and $J_n D=\bar{D}J_m$,
where $J_k$ is defined as in Lemma~\ref{rr2}.
\end{lemma}
\begin{proof}
By Lemma~\ref{lem2.12}, $A^C \in\C^{2m\times 2n}$ has full row rank $2m$.
Hence, there exists $D_0 \in \C^{2n\times 2m}$ such that $A^C D_0=I_{2m}$. Define
$D=\frac{1}{2}\left(D_0+J_n^{-1}\bar{D_0}J_m\right)$.
Using the identity $J_m A^C=\bar{(A^C)}J_n$, a direct computation shows
$A^C D=I_{2m}$. Moreover, by construction, $J_n D=\bar{D}J_m$.
\end{proof}

\begin{lemma}\label{rr1}
Let $A \in \mathbb{H}^{m \times n}$ be a quaternion matrix.
\begin{enumerate}[noitemsep]
\item[$(a)$] If $A$ has full column rank, i.e., $\rank(A) = n$ (with $m \geq n$), then there exists a left inverse $B \in \mathbb{H}^{n \times m}$ such that $BA = I_n$.
\item[$(b)$] If $A$ has full row rank, i.e., $\rank(A) = m$ (with $n \geq m$), then there exists a right inverse $B \in \mathbb{H}^{n \times m}$ such that $AB = I_m$.
\end{enumerate}
\end{lemma}
\begin{proof}
\begin{enumerate}
\item[$(a)$]  
Assume $\rank(A)=n$.
By Lemma~\ref{rr3}, there exists a matrix $C\in\C^{2n\times 2m}$ such that $C A^C=I_{2n}$
and
$J_n C=\bar{C}J_m$. By Lemma~\ref{rr2}, there exists a unique
$B\in\Q^{n\times m}$ such that $C=B^C$.
Using Lemma~\ref{pro2.11}, we obtain
$(BA)^C=B^C A^C =C A^C =I_{2n}=(I_n)^C$.
Consequently, it follows that $BA=I_n$.

\item[$(b)$] 
Assume $\rank(A)=m$.
By Lemma~\ref{rr4}, there exists a matrix
$D\in\mathbb{C}^{2n\times 2m}$ such that
$A^C D=I_{2m}$ and $J_n D=\bar{D}J_m$.
By Lemma~\ref{rr2}, there exists a unique
$B\in\Q^{n\times m}$ such that $D= B^C$.
Using Lemma~\ref{pro2.11}, we obtain
$(AB)^C=A^C B^C = A^C D=I_{2m}=(I_m)^C$.
Consequently, it follows that $AB=I_m$.
\end{enumerate}
\end{proof}


\subsection{Properties of range and null space of quaternion matrices}
Given a quaternion matrix $A \in \Q^{m \times n}$, its left- and right- range and null spaces satisfy key orthogonality properties. The following lemma formalizes these relationships, which will be instrumental in subsequent analysis.

\begin{lemma}\label{lem2.7}
For $A \in \Q^{m \times n}$, the following holds:
    \begin{enumerate}[noitemsep]
	\item[$(a)$] $\left(\mathcal{N}_l(A)_l^{\overset{\perp}{}}\right)_l^{\overset{\perp}{}}=\mathcal{N}_l(A)$, $\left(\mathcal{N}_r(A)_r^{\overset{\perp} 
             {}}\right)_r^{\overset{\perp}{}}=\mathcal{N}_r(A)$.
	\item[$(b)$]  $\left(\mathcal{R}_l(A)_l^{\overset{\perp}{}}\right)_l^{\overset{\perp}{}}=\mathcal{R}_l(A)$, $\left(\mathcal{R}_r(A)_r^{\overset{\perp} 
            {}}\right)_r^{\overset{\perp}{}}=\mathcal{R}_r(A)$.
            \item[$(c)$] $\mathcal{N}_l(A)=\mathcal{R}_l\left(A^*\right)_l^{\overset{\perp}{}}$.
	\item[$(d)$] 
    $\mathcal{N}_r(A)=\mathcal{R}_r\left(A^*\right)_r^{\overset{\perp}{}}$.
    \end{enumerate}
\end{lemma}

\begin{proof}
    The results in parts $(a)$ and $(b)$ follow directly from Definitions \ref{def2.2} and \ref{def2.6}. We proceed with the proof of part $(c)$, and part $(d)$ follows analogously.

    \noindent
$(c)~$ Suppose $x \in \mathcal{N}_l(A)$, meaning that $xA = 0$. To establish $x \in \mathcal{R}_l(A^*)_l^{\perp}$, we must show that for every $y \in \mathcal{R}_l(A^*)$, the inner product satisfies $\langle x, y \rangle_l = 0$. Further, since $y \in \mathcal{R}_l(A^*)$, there exists some $z \in \Q^{1 \times n}$ such that $y = zA^*$. Applying the definition of the inner product from Definition \ref{def2.5}, we obtain
 $\langle x, y \rangle_l = \langle x, zA^* \rangle_l = x (zA^*)^* = x A z^* = 0.$
    This confirms that $x \in \mathcal{R}_l(A^*)_l^{\perp}$, and thus we conclude that $\mathcal{N}_l(A) \subseteq \mathcal{R}_l(A^*)_l^{\perp}$.

Conversely, assume $x \in \mathcal{R}_l(A^*)_l^{\perp}$. Then, for all $y \in \mathcal{R}_l(A^*)$, we have $\langle x, y \rangle_l = 0$. As $y = zA^*$ for some $z \in \Q^{1 \times n}$, it follows that $\langle x, y \rangle_l = xy^* = x(zA^*)^* = (xA)z^* = 0.$
Since this holds for all such $z$, we conclude that $xA = 0$, meaning that $x \in \mathcal{N}_l(A)$. Therefore, we have $\mathcal{R}_l(A^*)_l^{\perp} \subseteq \mathcal{N}_l(A)$, which completes the proof.
\end{proof}

The dimensions of the left and right range and null spaces of a quaternion matrix are fundamentally related. These relationships are analogous to the rank-nullity theorem in classical linear algebra. The following lemma formalizes these relationships.
\begin{lemma}\label{lem2.8}
	For $A \in \Q^{m \times n}$, the following holds:
	\begin{enumerate}[noitemsep]
            \item[$(a)$]  $\dim(\mathcal{R}_l(A))+\dim(\mathcal{N}_l(A))=m$.
		\item[$(b)$]  $\dim(\mathcal{R}_r(A))+\dim(\mathcal{N}_r(A))=n$.
	\end{enumerate}
\end{lemma}
\begin{proof}
    We will prove part $(a)$. The proof for part $(b)$ follows similarly.

\noindent
$(a)~$   
    Assume that the first $r$ rows of $A$ are left linearly independent, such that $\dim(\mathcal{R}_l(A)) = r$. Then, we can express $A$ as
$A = \begin{bmatrix} A_1^T & A_2^T \end{bmatrix}^T$,
    where $A_1 \in \Q^{r \times n}$ consists of $r$ left linearly independent rows, and $A_2 \in \Q^{(m-r) \times n}$ contains rows that are linear combinations of those in $A_1$. This implies that $A_2 = B A_1,$ \text{for some }$B \in \Q^{(m-r) \times r}.$
     Now, define the matrix
$X = \begin{bmatrix} -B & I_{m-r} \end{bmatrix} \in \Q^{(m-r) \times m}$.
    We then compute
$XA = \begin{bmatrix} -B & I_{m-r} \end{bmatrix} \begin{bmatrix} A_1 \\ B A_1 \end{bmatrix} = 0$.
     Let $\text{row}_i(X)$ denote the $i^{\text{th}}$ row of $X$. Then
    \begin{equation*}
        \begin{bmatrix} \text{row}_1(X)  \\ \vdots \\ \text{row}_{(m-r)}(X) \end{bmatrix} A = 0,
    \end{equation*}
     indicating that each row of $X$ is a solution to $xA = 0$. Therefore, all rows of $X$ lie in $\mathcal{N}_l(A)$. Next, we show that the rows of $X$ are left linearly independent. Suppose that there exists a vector $v \in \Q^{1 \times (m-r)}$ such that
    $vX = 0.$  Thus, we can write $v \begin{bmatrix} -B & I_{m-r} \end{bmatrix} = 0$ which implies $\begin{bmatrix} -vB & v \end{bmatrix} = 0$.  This implies $v = 0.$ Thus, the rows of $X$ are left linearly independent.  Now, we claim that any solution to $xA = 0$ can be written as a linear combination of the rows of $X$. Let $u = \begin{bmatrix} u_1 & u_2 \end{bmatrix}$, where $u_1 \in \Q^{1 \times r}$ and $u_2 \in \Q^{1 \times (m-r)}$. The equation $uA = 0$ leads to
    \begin{equation*}
        \begin{bmatrix} u_1 & u_2 \end{bmatrix} \begin{bmatrix} A_1 \\ B A_1 \end{bmatrix} = 0, \text{which implies}~ \left( u_1 + u_2 B \right) A_1 = 0.
    \end{equation*}
    Given that $A_1$ has full row rank, it admits a right inverse (see Lemma~\ref{rr1}).
Post-multiplying both sides by a right inverse of $A_1$, we obtain $u_1 + u_2 B = 0$, and hence $u_1 = -u_2 B$. Consequently, the solution vector can be represented by
$u = \begin{bmatrix} u_1 & u_2 \end{bmatrix} = \begin{bmatrix} -u_2 B & u_2 \end{bmatrix} = u_2 X.$ This establishes that any $u \in \mathcal{N}_l(A)$ belongs to row space of $X$. Because the rows of $X$ are linearly independent and span $\mathcal{N}_l(A)$, it follows that $\dim(\mathcal{N}_l(A)) = m - r.$
    Since we assumed that $\dim(\mathcal{R}_l(A)) = r$, we conclude that
$\dim(\mathcal{R}_l(A)) + \dim(\mathcal{N}_l(A)) = r + (m - r) = m$.
This completes the proof.
\end{proof}
The following lemmas establish the fundamental conditions under which the left and right range and null spaces of one quaternion matrix are contained within those of the other. These results will be used later to develop the main results of this study.
\begin{lemma}\label{lem2.9}
	Let $X \in \Q^{n \times m}$, $S \in \Q^{k \times m}$, and $T \in \Q^{n \times l}$. Then
	\begin{enumerate}[noitemsep]
		\item[$(a)$] $\mathcal{R}_l(X) \subseteq \mathcal{R}_l(S)$ if and only if there exists $U \in \Q^{n \times k}$ such that $X=US.$
		\item[$(b)$] $\mathcal{N}_l(T) \subseteq \mathcal{N}_l(X)$ if and only if there exists $V \in \Q^{l \times m}$ such that $X=TV.$
	\end{enumerate}
\end{lemma}
\begin{proof}
\begin{enumerate}
\item[$(a)$] Let $\text{row}_i(X)$ denote the $i^{\text{th}}$ row of $X$ for $i=1,2,\dots,n$, and let $\text{row}_i(S)$ denote the $i^{\text{th}}$ row of $S$ for $i=1,2,\dots,k$. If $\mathcal{R}_l(X) \subseteq \mathcal{R}_l(S)$, this implies  $\text{row}_i(X) \in \mathcal{R}_l(S), ~\text{for all} ~i=1,2,\dots,n,$  then each row of $X$ can be expressed as a linear combination of the rows of $S$. Thus, there exist scalars $u_{ij} \in \Q$ such that
$\text{row}_i(X) = \sum_{j=1}^{k} u_{ij} \text{row}_j(S)$, for all $i=1,2,\dots,n$.
    Therefore, $X$ can be written as
    \begin{equation*}
        X = \begin{bmatrix} u_{11} & u_{12} & \cdots & u_{1k} \\ u_{21} & u_{22} & \cdots & u_{2k} \\ \vdots & \vdots &  & \vdots \\ u_{n1} & u_{n2} & \cdots & u_{nk} \end{bmatrix} S = U S, ~\text{for some matrix ~$U \in \Q^{n \times k}$.}
    \end{equation*}
    
    Conversely, if $X = US$, then for any $y \in \mathcal{R}_l(X)$, there exists $x \in \Q^{1 \times n}$ such that
$y = xX = x(US) = (xU)S \in \mathcal{R}_l(S),$
    which shows that $\mathcal{R}_l(X) \subseteq \mathcal{R}_l(S)$. This completes the proof of part $(a)$.
    
 \item[$(b)$] Assume that $\mathcal{N}_l(T) \subseteq \mathcal{N}_l(X)$. Based on the fundamental properties of orthogonal complements
$\mathcal{N}_l(X)_l^{\perp} \subseteq \mathcal{N}_l(T)_l^{\perp}$.
    By Lemma \ref{lem2.7}, we have $\mathcal{N}_l(X)_l^{\perp} = \mathcal{R}_l(X^*)$ and $\mathcal{N}_l(T)_l^{\perp} = \mathcal{R}_l(T^*)$, it follows that
$ \mathcal{R}_l(X^*) \subseteq \mathcal{R}_l(T^*)$.
    By part $(a)$, there exists $U \in \Q^{m \times l}$ such that $X^* = UT^*$. Taking the conjugate transpose, we obtain $X = TU^*$. Setting $V = U^* \in \Q^{l \times m}$, we get $X = TV$. This proves one direction.
    
    Conversely, if $X = TV$, then for any $x \in \mathcal{N}_l(T)$, we have
    $xT = 0$ which implies $xX = xTV = 0,$ which implies $x \in \mathcal{N}_l(X)$. Thus, $\mathcal{N}_l(T) \subseteq \mathcal{N}_l(X)$, completing the proof.
    \end{enumerate}
\end{proof}
\begin{lemma}\label{lem2.10}
	Let $X \in \Q^{n \times m}$, $S \in \Q^{n \times k}$, and $T \in \Q^{l \times m}$. Then
	\begin{enumerate}[noitemsep]
		\item[$(a)$] $\mathcal{R}_r(X) \subseteq \mathcal{R}_r(S)$ if and only if there exists $U \in \Q^{k \times m}$ such that $X=SU.$
		\item[$(b)$] $\mathcal{N}_r(T) \subseteq \mathcal{N}_r(X)$ if and only if there exists $V \in \Q^{n \times l}$ such that $X=VT.$
	\end{enumerate}
\end{lemma}
\begin{proof}
    The proof follows in a manner similar to that of Lemma \ref{lem2.9}.
\end{proof}

\section{Generalized inverse of quaternion matrices with prescribed range and/or null space}\label{sec3}
This section studies generalized inverses of quaternion matrices with prescribed range and null space constraints. We define these generalized inverses, establish their existence conditions, and derive their explicit representations. The following definition formalizes the outer inverses with the prescribed subspace constraints.
\begin{definition}\label{def3.2}
Let $A \in \Q^{m \times n}$, $S_1 \in \Q^{n \times p}$, $T_1 \in \Q^{q \times m}$, $S_2 \in \Q^{l \times m}$, and $T_2 \in \Q^{n \times t}$. An outer inverse of matrix $A$ with specific range and/or null space constraints is a solution of the matrix equation $XAX = X$. Additionally, $X$ must meet one of the conditions given below:
\begin{itemize}[noitemsep]
      \item[$(a)$] If $\mathcal{R}_r(X) = \mathcal{R}_r(S_1)$, then $X$ is denoted as $A^{(2)}_{r, (S_1, *)}$.
      \item[$(b)$]  If $\mathcal{N}_r(X) = \mathcal{N}_r(T_1)$, then $X$ is denoted by $A^{(2)}_{r, (*, T_1)}$.
      \item[$(c)$] If $\mathcal{R}_r(X) = \mathcal{R}_r(S_1)$ and $\mathcal{N}_r(X) = \mathcal{N}_r(T_1)$, then $X$ is denoted as $A^{(2)}_{r, (S_1, T_1)}$.
      \item[$(d)$]  If $\mathcal{R}_l(X) = \mathcal{R}_l(S_2)$, then $X$ is denoted by $A^{(2)}_{l, (S_2, *)}$.
      \item[$(e)$] If $\mathcal{N}_l(X) = \mathcal{N}_l(T_2)$, then $X$ is denoted as $A^{(2)}_{l, (*, T_2)}$.
	\item[$(f)$] If $\mathcal{R}_l(X) = \mathcal{R}_l(S_2)$ and $\mathcal{N}_l(X) = \mathcal{N}_l(T_2)$, then $X$ is denoted as $A^{(2)}_{l, (S_2, T_2)}$.
	\item[$(g)$] If $X$ satisfies both the right and left range and null space conditions, that is, $\mathcal{R}_r(X) = \mathcal{R}_r(S_1)$, $\mathcal{N}_r(X) = \mathcal{N}_r(T_1)$, $\mathcal{R}_l(X) = \mathcal{R}_l(S_2)$, and $\mathcal{N}_l(X) = \mathcal{N}_l(T_2)$, then $X$ is denoted as $A^{(2)}_{(S_1, T_1), (S_2, T_2)}$.
\end{itemize}
\end{definition}
Next, we define the $\{1,2\}$-inverse that satisfies specific subspace conditions.
\begin{definition}\label{def3.3}
Let $A \in \Q^{m \times n}$, $S_1 \in \Q^{n \times p}$, $T_1 \in \Q^{q \times m}$, $S_2 \in \Q^{l \times m}$, and $T_2 \in \Q^{n \times t}$. A $\left\{1,2\right\}$-inverse of matrix $A$ with specific range and null space constraints is a solution to the matrix equations $AXA=A$ and $XAX = X$. Additionally, $X$ must satisfy one of the following conditions:
\begin{itemize}[noitemsep]
    \item[$(a)$]  If $\mathcal{R}_r(X) = \mathcal{R}_r(S_1)$ and $\mathcal{N}_r(X) = \mathcal{N}_r(T_1)$, then $X$ is denoted as $A^{(1,2)}_{r, (S_1, T_1)}$.
    \item[$(b)$] If $\mathcal{R}_l(X) = \mathcal{R}_l(S_2)$ and $\mathcal{N}_l(X) = \mathcal{N}_l(T_2)$, then $X$ is denoted as $A^{(1,2)}_{l, (S_2, T_2)}$.
    \item[$(c)$] If $X$ satisfies both right and left range and null space conditions, that is, $\mathcal{R}_r(X) = \mathcal{R}_r(S_1)$, $\mathcal{N}_r(X) = \mathcal{N}_r(T_1)$, $\mathcal{R}_l(X) = \mathcal{R}_l(S_2)$, and $\mathcal{N}_l(X) = \mathcal{N}_l(T_2)$, then $X$ is denoted as $A^{(1,2)}_{(S_1, T_1), (S_2, T_2)}$.
 \end{itemize}
 \end{definition}
Let $A\left\{1\right\}$, $A\{2\}_{r, (S_1, *)}$, $A\{2\}_{r, (*, T_1)}$, $A\{2\}_{l, (S_2, *)}$, and $A\{2\}_{l, (*, T_2)}$ be the set of all generalized inverses $A^{(1)}$, $A^{(2)}_{r, (S_1, *)}$, $A^{(2)}_{r, (*, T_1)}$, $A^{(2)}_{l, (S_2, *)}$, and $A^{(2)}_{l, (*, T_2)}$, respectively.

We now present the necessary conditions for the existence of generalized inverse of a matrix with prescribed right range and/or null space, along with the corresponding expression for this inverse.
\begin{theorem}\label{thm3.4}
	Let $A \in \Q_{\nu}^{m \times n}$, $S_1 \in \Q^{n \times p}$, $T_1 \in \Q^{q \times m}$, and $X=S_1\left(T_1AS_1\right)^{(1)}T_1$, where $\left(T_1AS_1\right)^{(1)}$ is a fixed but arbitrary element of $\left(T_1AS_1\right)\left\{1\right\}$. Then
	\begin{enumerate}[noitemsep]
            \item[$(a)$] $X \in A\{1\}$ if and only if $\rank(T_1AS_1)=\nu$.
		\item[$(b)$] $X \in A\{2\}_{r, (S_1, *)}$ if and only if $\rank\left(T_1AS_1\right)=\rank\left(S_1\right)$.
		\item[$(c)$] $X \in A\{2\}_{r, (*, T_1)}$ if and only if $\rank\left(T_1AS_1\right)=\rank\left(T_1\right)$.
		\item[$(d)$] There exists a unique $X = A^{(2)}_{r,(S_1, T_1)}$ if and only if $\rank\left(T_1AS_1\right)=\rank\left(S_1\right)=\rank\left(T_1\right)$.
		\item[$(e)$] There exists a unique $X = A^{(1,2)}_{r,(S_1, T_1)}$ if and only if $\rank\left(T_1AS_1\right)=\rank\left(S_1\right)=\rank\left(T_1\right)=\nu$.
	\end{enumerate}
\end{theorem}
\begin{proof}
\begin{enumerate}
\item[$(a)$]Assume that $\rank\left(T_1AS_1\right) = \nu$. From Lemma \ref{lem2.12}, we obtain
$\nu=\rank\left(T_1AS_1\right) \leq \rank\left(AS_1\right) \leq \rank\left(A\right)=\nu$.
	This implies that $\rank\left(AS_1\right)=\rank\left(A\right)$. Additionally, from Definition \ref{def2.2}, we obtain $\mathcal{R}_r\left(AS_1\right) \subseteq \mathcal{R}_r\left(A\right).$
    Since both subspaces have the same dimensions, consequently 
    $\mathcal{R}_r\left(AS_1\right) = \mathcal{R}_r\left(A\right).$ Hence, $\mathcal{R}_r\left(A\right) \subseteq \mathcal{R}_r\left(AS_1\right)$. Thus, there exists a matrix $Y \in \Q^{p \times n}$ such that $A=AS_1Y.$
 Now consider the expression
	\begin{equation*}
		AXA=A\left(S_1\left(T_1AS_1\right)^{(1)}T_1\right)A =A\left(S_1\left(T_1AS_1\right)^{(1)}T_1\right)AS_1Y.
	\end{equation*}
    Rewriting, we get
    \begin{equation}\label{pf1.1}
        AXA= AS_1\left(T_1AS_1\right)^{(1)}\left(T_1AS_1\right)Y.
    \end{equation}
We have $\rank\left(T_1AS_1\right)=\rank\left(AS_1\right)=\nu$ and by Definition \ref{def2.2}, we have $\mathcal{N}_r\left(AS_1\right) \subseteq \mathcal{N}_r\left(T_1AS_1\right)$. Consequently, 
$\mathcal{N}_r\left(T_1AS_1\right) = \mathcal{N}_r\left(AS_1\right)$.
Thus, $\mathcal{N}_r\left(T_1AS_1\right) \subseteq \mathcal{N}_r\left(AS_1\right)$  implies the existence of a matrix $Z \in \Q^{m \times q}$ such that 
$
AS_1=ZT_1AS_1.
$
Using this, we can write
$AS_1=ZT_1AS_1=Z\left(T_1AS_1\right)\left(T_1AS_1\right)^{(1)}\left(T_1AS_1\right)=AS_1\left(T_1AS_1\right)^{(1)}\left(T_1AS_1\right)$.
Substituting this into equation \eqref{pf1.1}, we obtain
$AXA=AS_1\left(T_1AS_1\right)^{(1)}\left(T_1AS_1\right)Y = AS_1Y=A$.
Thus, $X \in A\{1\}$.

Conversely, assume $X \in A\{1\}$. We have
\begin{equation*}
A=AXA=AX\left(AXA\right)=A\left(S_1\left(T_1AS_1\right)^{(1)}T_1\right)A\left(S_1\left(T_1AS_1\right)^{(1)}T_1\right)A.
\end{equation*}
Using Lemma \ref{lem2.12}, it follows that
$\rank(A) \leq \rank\left(T_1AS_1\left(T_1AS_1\right)^{(1)}T_1A\right) \leq \rank\left(T_1AS_1\right)$.
Additionally, we have
$\rank\left(T_1AS_1\right) \leq \rank\left(AS_1\right) \leq \rank\left(A\right)$.
By combining these inequalities, we conclude that 
$\rank\left(T_1AS_1\right)=\rank(A)=\nu$. 
\item[$(b)$] Assume $\rank\left(T_1AS_1\right)=\rank\left(S_1
\right)$. This implies that $\dim\left(\mathcal{R}_r\left(T_1AS_1\right)\right)=\dim\left(\mathcal{R}_r\left(S_1\right)\right)$, and hence $\dim\left(\mathcal{N}_r\left(T_1AS_1\right)\right)=\dim\left(\mathcal{N}_r\left(S_1\right)\right)$. Since $\mathcal{N}_r\left(S_1\right) \subseteq \mathcal{N}_r\left(T_1AS_1\right)$, we conclude that $\mathcal{N}_r\left(S_1\right) = \mathcal{N}_r\left(T_1AS_1\right)$. Consequently, $\mathcal{N}_r\left(T_1AS_1\right) \subseteq \mathcal{N}_r\left(S_1\right)$, which implies that there exists a matrix $Z \in \Q^{n \times q}$ such that $S_1=ZT_1AS_1$. Thus,
\begin{equation}\label{pr1.2}
	S_1=ZT_1AS_1\left(T_1AS_1\right)^{(1)}T_1AS_1 = S_1\left(T_1AS_1\right)^{(1)} T_1AS_1.
\end{equation}
From this, we derive
$XAX=S_1\left(T_1AS_1\right)^{(1)}T_1AS_1\left(T_1AS_1\right)^{(1)}T_1 = S_1\left(T_1AS_1\right)^{(1)}T_1=X$.
Hence, $X \in A\{2\}$. 

As $X=S_1\left(T_1AS_1\right)^{(1)}T_1$, and using Definition \ref{def2.2}, we get
$\mathcal{R}_r(X)=\mathcal{R}_r\left(S_1\left(T_1AS_1\right)^{(1)}T_1\right) \subseteq \mathcal{R}_r\left(S_1\right)$.
From \eqref{pr1.2}, it follows that
$S_1=S_1\left(T_1AS_1\right)^{(1)}T_1AS_1 = XAS_1$.
By Definition \ref{def2.2}, we get
$\mathcal{R}_r(S_1) = \mathcal{R}_r\left(XAS_1\right) \subseteq \mathcal{R}_r(X)$.
Combining both results, we conclude $\mathcal{R}_r(X)=\mathcal{R}_r\left(S_1\right)$.

Conversely, assume $X \in A\{2\}$ and $\mathcal{R}_r(X)=\mathcal{R}_r\left(S_1\right)$. Then
\begin{equation*}
X=XAX=S_1\left(T_1AS_1\right)^{(1)}T_1AS_1\left(T_1AS_1\right)^{(1)}T_1.
\end{equation*}
By Lemma \ref{lem2.12}, we get
\begin{equation}\label{pr1.3}
	\rank\left(X\right) \leq \rank\left(T_1AS_1\left(T_1AS_1\right)^{(1)}T_1\right) \leq \rank\left(T_1AS_1\right) \leq \rank\left(S_1\right).
\end{equation}
Since $\mathcal{R}_r(X)=\mathcal{R}_r\left(S_1\right)$, we have $\dim\left(\mathcal{R}_r(X)\right)=\dim\left(\mathcal{R}_r(S_1)\right)$, implying $\rank\left(X\right)=\rank\left(S_1\right) $. Combining this with the previous inequality gives 
$\rank\left(X\right) \leq  \rank\left(T_1AS_1\right) \leq \rank\left(S_1\right) =\rank\left(X\right)$.
Therefore, $\rank\left(T_1AS_1\right)=\rank\left(S_1\right)$. 
\item[$(c)$] Assume that $\rank(T_1AS_1) = \rank(T_1)$. According to Definition \ref{def2.2}, we have 
$\mathcal{R}_r(T_1AS_1) \subseteq \mathcal{R}_r(T_1)$. 
As both subspaces have the same dimensions, consequently
$\mathcal{R}_r(T_1AS_1) = \mathcal{R}_r(T_1)$. 
Thus, $\mathcal{R}_r(T_1) \subseteq \mathcal{R}_r(T_1AS_1)$, implying the existence of a matrix $Z \in \Q^{p \times m}$ such that
$T_1 = T_1AS_1Z$. 
Consequently, we obtain
\begin{equation}\label{pr1.4}
    T_1 = (T_1AS_1)(T_1AS_1)^{(1)}T_1AS_1Z = T_1AS_1(T_1AS_1)^{(1)}T_1.
\end{equation}
Utilizing this result, we deduce
$XAX = S_1 (T_1AS_1)^{(1)} T_1AS_1 (T_1AS_1)^{(1)} T_1 
    = S_1 (T_1AS_1)^{(1)} T_1 = X$.
Hence, $X \in A\{2\}$. 

Because $X = S_1(T_1AS_1)^{(1)}T_1$, it follows that $\mathcal{N}_r(T_1) \subseteq \mathcal{N}_r(X)$. From \eqref{pr1.4}, we obtain
$T_1 = T_1AS_1 (T_1AS_1)^{(1)} T_1 = T_1AX$.
Thus, $\mathcal{N}_r(X) \subseteq \mathcal{N}_r(T_1)$. Combining both results, we conclude that
$\mathcal{N}_r(X) = \mathcal{N}_r(T_1)$. 

Conversely, assume that $X \in A\{2\}$ and $\mathcal{N}_r(X) = \mathcal{N}_r(T_1)$. Then,
\begin{equation*}
    X = XAX = S_1(T_1AS_1)^{(1)} T_1AS_1 (T_1AS_1)^{(1)} T_1.
\end{equation*}
Using Lemma \ref{lem2.12}, we obtain
\begin{equation}\label{pr1.5}
    \rank(X) \leq \rank(T_1AS_1(T_1AS_1)^{(1)}T_1) \leq \rank(T_1AS_1) \leq \rank(T_1).
\end{equation}
Since $\mathcal{N}_r(X) = \mathcal{N}_r(T_1)$, we have
$\dim(\mathcal{N}_r(X)) = \dim(\mathcal{N}_r(T_1))$. 
Applying Lemma \ref{lem2.8}, we conclude that
$\rank(X) = \rank(T_1)$.
Substituting this into \eqref{pr1.5}, we obtain
$\rank(X) \leq \rank(T_1AS_1) \leq \rank(T_1) = \rank(X)$.
Thus, we conclude that $\rank(T_1AS_1) = \rank(T_1)$. 
\item[$(d)$] By combining parts $(b)$ and $(c)$, we obtain that $X \in A\{2\}_{r,(S_1,T_1)}$ if and only if 
$\rank(T_1AS_1) = \rank(S_1) = \rank(T_1)$.
It remains to show that under this condition the solution is unique. 

Suppose that $X_1$ and $X_2$ are two solutions satisfying
\[
X_iAX_i = X_i,~ R_r(X_i)=R_r(S_1),~ N_r(X_i)=N_r(T_1), \text{~for~} i=1,2.
\]
Let $D := X_1 - X_2$. Clearly, every column of $X_i$ lies in $R_r(S_1)$, so that the columns of $D$ belong to $R_r(S_1)$, i.e. $R_r(D)\subseteq R_r(S_1)$. Moreover, since $N_r(X_1)=N_r(X_2)=N_r(T_1)$, it follows that $N_r(T_1)\subseteq N_r(D)$.  

Now, for any $s\in R_r(S_1)$, we have $s=X_i u$ for some $u$, because $R_r(X_i)=R_r(S_1)$. Using $X_iAX_i=X_i$, we obtain $(X_iA)s = X_iA(X_iu)=X_iu=s$.
 From \cite[Theorem~3.1]{MR2775924}, we know that
$AR_r(S_1) \oplus N_r(T_1) = \Q^{m\times 1}$.
Hence, every vector $y\in \Q^{m}$ can be written uniquely as $y=t+n$ with $t\in AR_r(S_1)$ and $n\in N_r(T_1)$.  

For such a decomposition, we compute
$Dy = Dt + Dn$.
Since $n\in N_r(T_1)\subseteq N_r(D)$, we have $Dn=0$. Thus $Dy=Dt$. Moreover, since $t\in AR_r(S_1)$, there exists $s\in R_r(S_1)$ such that $t=As$, and therefore
$Dy = Dt = D(As) = X_1As - X_2As$.
But as shown above, $(X_iA)s=s$ for all $s\in R_r(S_1)$, so $X_1As = s = X_2As$. Consequently, $Dt=0$. Since this holds for arbitrary $y\in \Q^m$, we conclude that $Dy=0$ identically, i.e. $D=0$. Hence $X_1=X_2$, proving uniqueness.
\item[$(e)$] The result directly follows from parts $(a)$ and $(d)$. 
\end{enumerate}
\end{proof}
Next, we present the conditions required for the existence of a generalized inverse of a matrix with a prescribed left range and/or null space. In addition, we provide the corresponding expression for this inverse.
\begin{theorem}\label{thm3.5}
	Let $A \in \Q_{\nu}^{m \times n}$, $S_2 \in \Q^{l \times m}$, $T_2 \in \Q^{n \times t}$, and $X=T_2\left(S_2AT_2\right)^{(1)}S_2$, where $\left(S_2AT_2\right)^{(1)}$ is a fixed but arbitrary element of $\left(S_2AT_2\right)\left\{1\right\}$. Then
	\begin{enumerate}[noitemsep]
            \item[$(a)$] $X \in A\{1\}$ if and only if $\rank(S_2AT_2)=\nu$.
		\item[$(b)$] $X \in A\{2\}_{l, (S_2, *)}$ if and only if $\rank\left(S_2AT_2\right)=\rank\left(S_2\right)$.
		\item[$(c)$] $X \in A\{2\}_{l, (*, T_2)}$if and only if $\rank\left(S_2AT_2\right)=\rank\left(T_2\right)$.
		\item[$(d)$] There exists a unique $X = A^{(2)}_{l,(S_2, T_2)}$ if and only if $\rank\left(S_2AT_2\right)=\rank\left(S_2\right)=\rank\left(T_2\right)$.
		\item[$(e)$] There exists a unique $X = A^{(1,2)}_{l,(S_2, T_2)}$ if and only if $\rank\left(S_2AT_2\right)=\rank\left(S_2\right)=\rank\left(T_2\right)=\nu$.
	\end{enumerate}
\end{theorem}
\begin{proof}
		\begin{enumerate}
        \item[$(a)$]
    Assume that $\rank\left(S_2AT_2\right) = \nu$. Using Lemma \ref{lem2.12}, we obtain
$\nu =\rank\left(S_2AT_2\right) \leq \rank\left(S_2A\right) \leq \rank\left(A\right)= \nu$.
	Thus, $\rank\left(S_2A\right)=\rank\left(A\right)$. Additionally, from Definition \ref{def2.2}, we have 
    $ \mathcal{R}_l\left(S_2A\right) \subseteq \mathcal{R}_l\left(A\right)$.
    Because both subspaces have the same dimensions, consequently 
    $ \mathcal{R}_l\left(S_2A\right) = \mathcal{R}_l\left(A\right)$.
    Hence, $\mathcal{R}_l\left(A\right) \subseteq \mathcal{R}_l\left(S_2A\right)$, implying the existence of a matrix $Y \in \Q^{m \times l}$ such that 
     $A=YS_2A$.
     
    Now, consider the expression
	$AXA=A\left(T_2\left(S_2AT_2\right)^{(1)}S_2\right)A =YS_2A\left(T_2\left(S_2AT_2\right)^{(1)}S_2\right)A$.
        Rewriting, we get
        \begin{equation}\label{pf2.1}
            AXA =Y\left(S_2AT_2\right)\left(S_2AT_2\right)^{(1)}S_2A.
        \end{equation}
		We have $\rank\left(S_2AT_2\right)=\rank\left(S_2A\right)$ and by Definition \ref{def2.2}, we have $\mathcal{N}_l\left(S_2A\right) \subseteq \mathcal{N}_l\left(S_2AT_2\right)$. Consequently, 
        $ \mathcal{N}_l\left(S_2AT_2\right) = \mathcal{N}_l\left(S_2A\right)$.
        Thus, $\mathcal{N}_l\left(S_2AT_2\right) \subseteq \mathcal{N}_l\left(S_2A\right)$, implying the existnece of a matrix $Z \in \Q^{t \times n}$ such that $S_2A=S_2AT_2Z$.
		Using this result,
$S_2A=S_2AT_2Z=\left(S_2AT_2\right)\left(S_2AT_2\right)^{(1)}\left(S_2AT_2\right)Z=\left(S_2AT_2\right)\left(S_2AT_2\right)^{(1)}S_2A$.
		Substituting this into equation \eqref{pf2.1}, we obtain
$AXA=Y\left(S_2AT_2\right)\left(S_2AT_2\right)^{(1)}S_2A= YS_2A=A$.
		Thus, $X \in A\{1\}$.
		
		Conversely, assume $X \in A\{1\}$. We have
		\begin{equation*}
A=AXA=AX\left(AXA\right)=A\left(T_2\left(S_2AT_2\right)^{(1)}S_2\right)A\left(T_2\left(S_2AT_2\right)^{(1)}S_2\right)A.
		\end{equation*}
		Using Lemma \ref{lem2.12}, it follows that
	$\rank(A) \leq \rank\left(S_2AT_2\left(S_2AT_2\right)^{(1)}S_2A\right) \leq \rank\left(S_2AT_2\right)$.
		Additionally, we have
			$\rank\left(S_2AT_2\right) \leq \rank\left(S_2A\right) \leq \rank\left(A\right)$.
		By combining these inequalities, we conclude that 
        $\rank\left(S_2AT_2\right)=\rank(A)=\nu$. 
	\item[$(b)$] Assume that $\rank\left(S_2AT_2\right)=\rank\left(S_2\right)$. This implies that $\dim\left(\mathcal{R}_l\left(S_2AT_2\right)\right)=\dim\left(\mathcal{R}_l\left(S_2\right)\right)$, and hence $\dim\left(\mathcal{N}_l\left(S_2AT_2\right)\right)=\dim\left(\mathcal{N}_l\left(S_2\right)\right)$. Because $\mathcal{N}_l\left(S_2\right) \subseteq \mathcal{N}_l\left(S_2AT_2\right)$, we can conclude that $\mathcal{N}_l\left(S_2\right) = \mathcal{N}_l\left(S_2AT_2\right)$. Consequently, $\mathcal{N}_l\left(S_2AT_2\right) \subseteq \mathcal{N}_l\left(S_2\right)$, which implies the existence of a matrix $Z \in \Q^{t \times m}$ such that $S_2=S_2AT_2Z$. Thus,
		\begin{equation}\label{pr2.2}
			S_2=S_2AT_2\left(S_2AT_2\right)^{(1)}S_2AT_2Z = S_2AT_2\left(S_2AT_2\right)^{(1)} S_2.
		\end{equation}
		From this, we derive
$XAX=T_2\left(S_2AT_2\right)^{(1)}S_2AT_2\left(S_2AT_2\right)^{(1)}S_2 = T_2\left(S_2AT_2\right)^{(1)}S_2=X$.
		Hence, $X \in A\{2\}$. 
        
        As $X=T_2\left(S_2AT_2\right)^{(1)}S_2$, and using Definition \ref{def2.2}, we get
$\mathcal{R}_l(X)=\mathcal{R}_l\left(T_2\left(S_2AT_2\right)^{(1)}S_2\right) \subseteq \mathcal{R}_l\left(S_2\right)$.
		From \eqref{pr2.2}, it follows that
$S_2=S_2AT_2\left(S_2AT_2\right)^{(1)}S_2 = S_2AX$.
		By Definition \ref{def2.2}, we get
$\mathcal{R}_l(S_2) = \mathcal{R}_l\left(S_2AX\right) \subseteq \mathcal{R}_l(X)$.
		Combining these results, we can conclude that $\mathcal{R}_l(X)=\mathcal{R}_l\left(S_2\right)$. 
		
		Conversely, assume $X \in A\{2\}$ and $\mathcal{R}_l(X)=\mathcal{R}_l\left(S_2\right)$. Then,
		\begin{equation*}
X=XAX=T_2\left(S_2AT_2\right)^{(1)}S_2AT_2\left(S_2AT_2\right)^{(1)}S_2.       \end{equation*}
		By Lemma \ref{lem2.12}, we get
		\begin{equation}\label{pr2.3}
			\rank\left(X\right) \leq \rank\left(S_2AT_2\left(S_2AT_2\right)^{(1)}S_2\right) \leq \rank\left(S_2AT_2\right) \leq \rank\left(S_2\right).
		\end{equation}
		Since $\mathcal{R}_l(X)=\mathcal{R}_l\left(S_2\right)$, we obtain $\dim\left(\mathcal{R}_l(X)\right)=\dim\left(\mathcal{R}_l(S_2)\right)$, which implies $\rank\left(X\right)=\rank\left(S_2\right) $. Combining this with \eqref{pr2.3}, we obtain
		$\rank\left(X\right) \leq  \rank\left(S_2AT_2\right) \leq \rank\left(S_2\right) =\rank\left(X\right)$.
Thus, we conclude that $\rank\left(S_2AT_2\right)=\rank\left(S_2\right)$.   
		\item[$(c)$] Assume that $\rank\left(S_2AT_2\right)=\rank\left(T_2\right)$. By Definition \ref{def2.2}, we have $\mathcal{R}_l\left(S_2AT_2\right) \subseteq \mathcal{R}_l(T_2)$. As both subspaces have the same dimensions, it follows that $\mathcal{R}_l\left(S_2AT_2\right) = \mathcal{R}_l(T_2)$. Thus, $\mathcal{R}_l\left(T_2\right) \subseteq \mathcal{R}_l(S_2AT_2)$,  implying the existence of a matrix $Z \in \Q^{n \times l}$ such that $T_2=ZS_2AT_2$. Using this result, we obtain
		\begin{equation}\label{pr2.4}
T_2=Z\left(S_2AT_2\right)\left(S_2AT_2\right)^{(1)}S_2AT_2=T_2\left(S_2AT_2\right)^{(1)}S_2AT_2.
		\end{equation}
		Utilizing this, we derive
$XAX=T_2\left(S_2AT_2\right)^{(1)}S_2AT_2\left(S_2AT_2\right)^{(1)}S_2 = T_2\left(S_2AT_2\right)^{(1)}S_2=X$.
Hence, $X \in A\{2\}$.

        Since $X=T_2\left(S_2AT_2\right)^{(1)}S_2$, it follows that $\mathcal{N}_l(T_2) \subseteq \mathcal{N}_l(X)$. From \eqref{pr2.4}, we also have
$T_2=T_2\left(S_2AT_2\right)^{(1)}S_2AT_2 = XAT_2$.
		Thus, $\mathcal{N}_l(X) \subseteq \mathcal{N}_l\left(T_2\right)$. Combining these results, we conclude that $\mathcal{N}_l(X) = \mathcal{N}_l\left(T_2\right)$.
		
		Conversely, assume $X \in A\{2\}$ and $\mathcal{N}_l(X)=\mathcal{N}_l\left(T_2\right)$. Then,
		\begin{equation*}
X=XAX=T_2\left(S_2AT_2\right)^{(1)}S_2AT_2\left(S_2AT_2\right)^{(1)}S_2.
		\end{equation*}
		Using Lemma \ref{lem2.12}, we obtain
		\begin{equation}\label{pr2.5}
			\rank(X) \leq \rank\left(S_2AT_2\left(S_2AT_2\right)^{(1)}S_2\right) \leq \rank\left(S_2AT_2\right) \leq \rank(T_2).
		\end{equation}
		Since $\mathcal{N}_l(X)=\mathcal{N}_l\left(T_2\right)$, we get $\dim\left(\mathcal{N}_l(X)\right)=\dim\left(\mathcal{N}_l\left(T_2\right)\right)$. Applying Lemma \ref{lem2.8}, we conclude that $\rank(X)=\rank\left(T_2\right)$. Substituting this into \eqref{pr2.5}, we obtain
$\rank(X) \leq \rank\left(S_2AT_2\right) \leq \rank\left(T_2\right)=\rank(X)$.
		Thus, we conclude that $\rank\left(S_2AT_2\right)=\rank(T_2)$. 
		\item[$(d)$]  By combining parts $(b)$ and $(c)$, we obtain that $X \in A\{2\}_{l,(S_2,T_2)}$ if and only if 
$\rank(S_2AT_2) = \rank(S_2) = \rank(T_2)$.
It remains to show that under this condition the solution is unique.
Suppose that $X_1$ and $X_2$ are two solutions satisfying
\[
X_iAX_i = X_i, ~ R_l(X_i)=R_l(S_2), ~ N_l(X_i)=N_l(T_2), \text{~for~} i=1,2.
\]
Let $D := X_1 - X_2$. Clearly, every row of $X_i$ lies in $R_l(S_2)$, so that the rows of $D$ belong to $R_l(S_2)$, i.e. $R_l(D)\subseteq R_l(S_2)$. Moreover, since $N_l(X_1)=N_l(X_2)=N_l(T_2)$, it follows that $N_l(T_2)\subseteq N_l(D)$.  

Now, for any $s\in R_l(S_2)$, we have $s=vX_i$ for some $v$, because $R_l(X_i)=R_l(S_2)$. Using $X_iAX_i=X_i$, we obtain
$sA X_i = (vX_i)A X_i = v(X_iAX_i) = vX_i = s$.
From \cite[Theorem~3.1]{MR2775924}, we know that
$R_l(S_2)A \oplus N_l(T_2) = \Q^{1\times n}$.
Hence, every row vector $y\in \Q^{1\times n}$ can be written uniquely as $y=u+w$ with $u\in R_l(S_2)A$ and $w\in N_l(T_2)$.  For such a decomposition, we compute
$yD = uD + wD$.
Since $w\in N_l(T_2)\subseteq N_l(D)$, we have $wD=0$. Thus $yD=uD$. Moreover, since $u\in R_l(S_2)A$, there exists $s\in R_l(S_2)$ such that $u=sA$, and therefore
$yD=uD = (sA)D=sAX_1-sAX_2$.
But as shown above, $sAX_i=s$ for all $s\in R_l(S_2)$, so
$sAX_1 = s=sAX_2$. Consequently, $yD=0$.
Since this holds for arbitrary $y\in \Q^{1\times n}$, we conclude that $yD=0$, which implies $D=0$. Hence, $X_1=X_2$, proving uniqueness.
		\item[$(e)$] The result directly follows from parts $(a)$ and $(d)$.
        \end{enumerate}
\end{proof}
We now present the conditions required for the existence of a generalized inverse of a matrix with prescribed right and left range and null space. We also provide a corresponding expression for this inverse.
\begin{theorem}\label{thm3.6}
	Let $A \in \Q_{\nu}^{m \times n}$, $S \in \Q^{n \times m}$, $T \in \Q^{n \times m}$, and $X=S\left(TAS\right)^{(1)}T$, where $\left(TAS\right)^{(1)}$ is a fixed but arbitrary element of $\left(TAS\right)\left\{1\right\}$. Then
	\begin{enumerate}[noitemsep]
	    \item[$(a)$] There exists a unique $X=A^{(2)}_{(S, T), (T, S)}$ if and only if $\rank\left(TAS\right)=\rank\left(S\right)=\rank\left(T\right)$.
		\item[$(b)$] There exists a unique $X=A^{(1,2)}_{(S, T), (T, S)}$ if and only if $\rank\left(TAS\right)=\rank\left(S\right)=\rank\left(T\right)=\nu$.
	\end{enumerate}
\end{theorem}
\begin{proof}
	The proof for both parts follows directly from Theorems \ref{thm3.4} and \ref{thm3.5}.
\end{proof}
The following theorems establish an explicit formula for the Moore–Penrose inverse, Drazin inverse, and group inverse of a quaternion matrix in terms of its outer inverse under predefined subspace conditions.
\begin{theorem}\label{thm3.7}
	Let $A \in \Q^{m \times n}$. Then, the Moore-Penrose inverse $A^{\dagger}$ is given by
$A^{\dagger} = A^{(2)}_{l,\left(A^*, A^*\right)}=A^{(2)}_{r,\left(A^*, A^*\right)}=A^*\left(A^*AA^*\right)^{(1)}A^*$,
where $\left(A^*AA^*\right)^{(1)} \in \left(A^*AA^*\right)\{1\}$. 
\end{theorem}
\begin{proof}
We prove the result for $A^{\dagger} = A^{(2)}_{l, (A^*, A^*)}$. The proof for $A^{\dagger} = A^{(2)}_{r, (A^*, A^*)}$ follows similarly.
 Evidently, $A^{\dagger} = A^{(2)}$. Thus, it remains to show that
  $\mathcal{R}_l(A^{\dagger}) = \mathcal{R}_l(A^*) ~ \text{and} ~ \mathcal{N}_l(A^{\dagger}) = \mathcal{N}_l(A^*)$.
 Using the properties of the Moore-Penrose inverse, we have
 $A^{\dagger} = A^{\dagger} A A^{\dagger} = A^{\dagger} (A A^{\dagger})^*  = A^{\dagger} (A^{\dagger})^* A^*$ and $A^* = (A A^{\dagger} A)^* = A^* (A^{\dagger})^* A^* = A^* (A A^{\dagger})^* = A^* A A^{\dagger}$.
  Using the above identities and Definition \ref{def2.2}, we get 
$\mathcal{R}_l(A^{\dagger}) \subseteq \mathcal{R}_l(A^*)$ and $\mathcal{R}_l(A^*) \subseteq \mathcal{R}_l(A^{\dagger})$.
    Consequently,
   $\mathcal{R}_l(A^{\dagger}) = \mathcal{R}_l(A^*)$.
   
Next, we show that the left null spaces are also equal. Let $u \in \mathcal{N}_l(A^*)$, such that $u A^* = 0$. Then
 $u A^* (A^{\dagger})^* = 0 \implies u (A^{\dagger} A)^* = 0 \implies u A^{\dagger} A = 0 \implies u A^{\dagger} A A^{\dagger} = 0 \implies u A^{\dagger} = 0$.
    Thus, $u \in \mathcal{N}_l(A^{\dagger})$, and hence,
  $\mathcal{N}_l(A^*) \subseteq \mathcal{N}_l(A^{\dagger})$.
Conversely, let $v \in \mathcal{N}_l(A^{\dagger})$, so that $v A^{\dagger} = 0$. Then
$v A^{\dagger} A = 0 \implies v (A^{\dagger} A)^* = 0 \implies v A^* (A^{\dagger})^* A^* = 0 \implies v (A A^{\dagger} A)^* = 0 \implies v A^* = 0$.
    Thus, $v \in \mathcal{N}_l(A^*)$, hence, $\mathcal{N}_l(A^{\dagger}) \subseteq \mathcal{N}_l(A^*)$. Consequently,
   $\mathcal{N}_l(A^{\dagger}) = \mathcal{N}_l(A^*)$.
Therefore, $A^{\dagger} = A^{(2)}_{l, (A^*, A^*)}$.  
This completes the proof.
\end{proof}
\begin{theorem}\label{thm3.7d}
    Let $A \in \Q^{m \times m}$ with $\text{Ind}(A) = k$. Then, the Drazin inverse of $A$ is given by  
$A^{D} = A^{(2)}_{l,\left(A^k, A^k\right)} = A^{(2)}_{r,\left(A^k, A^k\right)}$.
\end{theorem}
\begin{proof}
    We aim to prove that  
  $A^D = A^{(2)}_{l,\left(A^k, A^k\right)}$. 
    Because $A^D$ is an outer inverse, it suffices to establish that  
     $\mathcal{R}_l(A^D) = \mathcal{R}_l(A^k)$ and $\mathcal{N}_l(A^D) = \mathcal{N}_l(A^k)$.
 From the defining properties of the Drazin inverse, we have 
    \begin{equation}\label{eqdra1}
        A^D = A^D A A^D, ~ AA^D = A^D A, ~ A^{k+1} A^D = A^k.
    \end{equation}
    Using these identities, it follows that
    \begin{equation}\label{eqdra2}
        A^k = A^D A^{k+1}, ~ A^D = A^k (A^D)^{k+1}, ~ A^D = (A^D)^{k+1} A^k.
    \end{equation}
    We now verify the range and null space conditions. From Definition~\ref{def2.2} and using \eqref{eqdra1} and \eqref{eqdra2}, we obtain
 $\mathcal{R}_l(A^D) = \mathcal{R}_l(A^k)$.
    Let $x \in \mathcal{N}_l(A^D)$, that is, $x A^D = 0$. Post-multiplying by $A^{k+1}$ yields  
    $ x A^D A^{k+1} = 0$.
    Using \eqref{eqdra2}, this simplifies to $x A^k = 0$,  implying that $x \in \mathcal{N}_l(A^k)$. Hence,  
  $\mathcal{N}_l(A^D) \subseteq \mathcal{N}_l(A^k)$.
  Conversely, let $y \in \mathcal{N}_l(A^k)$, that is, $y A^k = 0$. Post-multiplying by $(A^D)^{k+1}$ yields  
    $y A^k (A^D)^{k+1} = 0$.
    Again using \eqref{eqdra2}, we obtain $y A^D = 0$, implying $y \in \mathcal{N}_l(A^D)$. Thus,  
  $\mathcal{N}_l(A^k) \subseteq \mathcal{N}_l(A^D)$.
    This establishes $\mathcal{N}_l(A^D) = \mathcal{N}_l(A^k)$. Therefore, $A^D = A^{(2)}_{l,\left(A^k, A^k\right)}$. 

    The proof of $A^D = A^{(2)}_{r,\left(A^k, A^k\right)}$ follows analogously. Hence, the result is proved.
\end{proof}
\begin{corollary}\label{cord}
    If $\text{Ind}(A) = 1$, then the group inverse of $A \in \Q^{m \times m}$ is given by  $A^{\#} = A^{(2)}_{l,\left(A, A\right)} = A^{(2)}_{r,\left(A, A\right)}$.
\end{corollary}
\begin{proof}
   The proof follows immediately by setting $k = 1$ in Theorem \ref{thm3.7d}.
\end{proof}
The following theorems establish the key results regarding the range and null space of a quaternion matrix based on its full rank decomposition.
\begin{theorem}\label{thm3.8}
    Let $A \in \Q_{\nu}^{m \times n}$, and suppose $A = F_1 G_1$ is an arbitrary full rank decomposition of $A$. Then,
$ \mathcal{R}_r(A) = \mathcal{R}_r(F_1)$ and $\mathcal{N}_r(A) = \mathcal{N}_r(G_1)$.
   \end{theorem}
\begin{proof}
Let $\text{col}_i(A)$ and $\text{col}_i(F_1)$ denote the $i^{\text{th}}$ column of matrix $A$ and matrix $F_1$, respectively. Consider $G_1=(g_{ij})_{\nu \times n}$. We have
    \begin{equation*}
        \begin{bmatrix}
            \text{col}_1(A) & \text{col}_2(A) & \cdots & \text{col}_n(A)
        \end{bmatrix} = \begin{bmatrix}
            \text{col}_1(F_1) & \text{col}_2(F_1) & \cdots & \text{col}_{\nu}(F_1)
        \end{bmatrix}
        \begin{bmatrix}
            g_{11} & g_{12} & \cdots & g_{1n} \\
            g_{21} & g_{22} & \cdots & g_{2n} \\
            \vdots & \vdots &        & \vdots\\
            g_{\nu 1} & g_{\nu 2} & \cdots & g_{\nu n}
        \end{bmatrix}.
    \end{equation*}
 Therefore,
$\text{col}_i(A) = \text{col}_1(F_1) g_{1i} + \text{col}_2(F_1) g_{2i} + \cdots + \text{col}_{\nu}(F_1) g_{\nu i}$, for $i=1,2,\ldots,n$. Thus,
$\text{col}_i(A) \in \mathcal{R}_r(F_1)$, for $i=1,2,\ldots,n$.
    Consequently,
  $\mathcal{R}_r(A) \subseteq \mathcal{R}_r(F_1)$.
    Since $\nu = \rank(A) = \dim(\mathcal{R}_r(A)) = \dim(\mathcal{R}_r(F_1))$, it follows that
   $\mathcal{R}_r(A) = \mathcal{R}_r(F_1)$.
   
    Next, we investigate the relationship between the null spaces. Let $x \in \mathcal{N}_r(G_1)$, implying that $G_1 x = 0$. Consequently, $F_1 G_1 x = 0$, implying 
that $A x = 0$.
    Thus, $x \in \mathcal{N}_r(A)$, leading to
   $\mathcal{N}_r(G_1) \subseteq \mathcal{N}_r(A)$.
    Furthermore, we have
$\dim(\mathcal{N}_r(G_1)) = n - \dim(\mathcal{R}_r(G_1)) = n - \rank(G_1) = n - \nu$ and
    $\dim(\mathcal{N}_r(A)) = n - \dim(\mathcal{R}_r(A)) = n - \rank(A) = n - \nu.$
    Since $\dim(\mathcal{N}_r(G_1)) = \dim(\mathcal{N}_r(A))$, we conclude that
   $\mathcal{N}_r(A) = \mathcal{N}_r(G_1)$.
    This completes the proof.
\end{proof}
\begin{theorem}\label{thm3.9}
Let $A \in \Q_{\nu}^{m \times n}$, and suppose $A = G_2 F_2$ is an arbitrary full rank decomposition of $A$. Then,
$\mathcal{R}_l(A)=\mathcal{R}_l(F_2)$ and $\mathcal{N}_l(A) = \mathcal{N}_l(G_2)$.
\end{theorem}
\begin{proof}
  Let $\text{row}_i(A)$ and $\text{row}_i(F_2)$ denote the $i^{\text{th}}$ row of matrix $A$ and matrix $F_2$, respectively. Consider $G_2=(g_{ij})_{m \times \nu}$. We have
\begin{equation*}
    \begin{bmatrix}
        \text{row}_1(A) \\ \text{row}_2(A) \\ \vdots \\ \text{row}_m(A)
    \end{bmatrix} = \begin{bmatrix}
        g_{11} & g_{12} & \cdots & g_{1 \nu} \\
        g_{21} & g_{22} & \cdots & g_{2 \nu} \\
        \vdots & \vdots &        & \vdots\\
        g_{m1} & g_{m2} & \cdots & g_{m \nu}
    \end{bmatrix}\begin{bmatrix}
        \text{row}_1(F_2) \\ \text{row}_2(F_2) \\ \vdots \\ \text{row}_{\nu}(F_2)
    \end{bmatrix}.
\end{equation*}
Therefore, $\text{row}_i(A) =  g_{i1} \text{row}_1(F_2) + g_{i2} \text{row}_2(F_2) + \cdots + g_{i \nu} \text{row}_{\nu}(F_2)$, for $i=1,2,\ldots,m$.
Thus, $\text{row}_i(A) \in \mathcal{R}_l(F_2)$, for $i=1,2,\ldots,m$.
Consequently,
$\mathcal{R}_l(A) \subseteq \mathcal{R}_l(F_2)$.
Since $\nu = \rank(A) = \dim(\mathcal{R}_l(A)) = \dim(\mathcal{R}_l(F_2))$, it follows that
$\mathcal{R}_l(A) = \mathcal{R}_l(F_2)$.

Next, we investigate the relationship between the null spaces. Let $x \in \mathcal{N}_l(G_2)$, implying that $x G_2 = 0$. Consequently,
$x G_2 F_2 = 0$, implying 
that $x A = 0$.
Thus, $x \in \mathcal{N}_l(A)$, leading to
$\mathcal{N}_l(G_2) \subseteq \mathcal{N}_l(A)$.
Furthermore, we have
$ \dim(\mathcal{N}_l(G_2)) = m - \dim(\mathcal{R}_l(G_2)) = m - \rank(G_2) = m - \nu$ and
$\dim(\mathcal{N}_l(A)) = m - \dim(\mathcal{R}_l(A)) = m - \rank(A) = m - \nu$.
As $\dim(\mathcal{N}_l(G_2)) = \dim(\mathcal{N}_l(A))$, we conclude that
$\mathcal{N}_l(A) = \mathcal{N}_l(G_2)$.
This completes the proof.  
\end{proof}
The following theorem establishes the explicit representation of $A^{(2)}_{r,\left(W_1, W_1\right)}$, based on the full rank factorization of $W_1$.
\begin{theorem}\label{thm3.10}
Let $A \in \Q_{\nu}^{m \times n}$ and $W_1 \in \Q^{n \times m}$ be a matrix such that $\dim(\mathcal{R}_r(W_1)) = s \leq \nu$ and $\dim(\mathcal{N}_r(W_1)) = m-s$. Assume $W_1 = S_1 T_1$ is a full rank factorization of $W_1$, and define
$ X = S_1 (T_1 A S_1)^{-1} T_1$.
If $A$ has a $\{2\}$-inverse $A^{(2)}_{r,(W_1,W_1)}$, then the following statements hold:
\begin{enumerate}[noitemsep]
    \item[$(a)$] $T_1 A S_1$ is an invertible matrix.
    \item[$(b)$] $A^{(2)}_{r,(W_1,W_1)} = X = A^{(2)}_{r,(S_1,T_1)}$.
\end{enumerate}
\end{theorem}
\begin{proof}
\begin{enumerate}
\item[$(a)$]
According to Theorem \ref{thm3.8}, we have the following relations:
\begin{equation}\label{3.10.1}
    \mathcal{R}_r(W_1)=\mathcal{R}_r(S_1), ~ \mathcal{N}_r(W_1)=\mathcal{N}_r(T_1).
\end{equation}
Now, observe that $\mathcal{R}_r(AS_1) = A \mathcal{R}_r(S_1) = A \mathcal{R}_r(W_1)$.
Using \cite[Theorem $3.1$]{MR2775924}, we get 
\begin{equation}\label{3.10.2}
A \mathcal{R}_r(W_1) \oplus \mathcal{N}_r(W_1) = \Q^{m \times 1}.
\end{equation}
Utilizing this, we compute
$\dim(\mathcal{R}_r(AS_1)) = \dim(A \mathcal{R}_r(W_1))
     = m- \dim(\mathcal{N}_r(W_1)) = m -(m-s) = s$.
Therefore, $\rank(AS_1) = s.$ By a similar argument, we also obtain
$\rank(AW_1) = s.$ This implies that
$AW_1 = AS_1 T_1$ is a full rank factorization of $AW_1$.

Next, suppose $y \in \Q^{s \times 1}$ satisfies $(T_1 AS_1) y = 0$. This implies that $AS_1 y \in \mathcal{N}_r(T_1)$.
Using equation \eqref{3.10.1}, we obtain $AS_1 y \in \mathcal{N}_r(W_1).$
 In addition, because $AS_1 y \in \mathcal{R}_r(AS_1) = A \mathcal{R}_r(W_1)$, we have
$AS_1y \in A \mathcal{R}_r(W_1) \cap \mathcal{N}_r(W_1)$.
Using equation \eqref{3.10.2}, it follows that
$AS_1 y = 0$, which implies $y \in \mathcal{N}_r(AS_1)$.
Now, observe that
$\dim(\mathcal{N}_r(AS_1)) = s - \dim(\mathcal{R}_r(AS_1)) = s - \rank(AS_1) = 0$,
Therefore, we conclude that $y = 0.$
Hence, for any $y \in \Q^{s \times 1}$,
$(T_1 AS_1) y = 0$ implies $y = 0$.
Therefore,
$\dim(\mathcal{N}_r(T_1 AS_1)) = 0$,
and consequently,
$\rank(T_1 AS_1) = \dim(\mathcal{R}_r(T_1 AS_1)) = s - \dim(\mathcal{N}_r(T_1 AS_1)) = s$,
implying that $T_1 AS_1$ is a nonsingular matrix; hence, it is invertible. 
\item[$(b)$]
We show that $X$ satisfies $\{2\}$-inverse property:
$XAX = S_1 (T_1 AS_1)^{-1} T_1 A S_1 (T_1 AS_1)^{-1} T_1 = S_1 (T_1 AS_1)^{-1} T_1 = X$,
thus, $X \in A\{2\}.$ Applying Definition \ref{def2.2}, we obtain the following results
\begin{equation}\label{3.10.3}
    \mathcal{R}_r(X) = \mathcal{R}_r \left( S_1 (T_1 AS_1)^{-1} T_1 \right) \subseteq \mathcal{R}_r(S_1),  ~\mathcal{N}_r(T_1) \subseteq \mathcal{N}_r \left( S_1 (T_1 AS_1)^{-1} T_1 \right) = \mathcal{N}_r(X).
\end{equation}
We will now prove that $\mathcal{R}_r(S_1) \subseteq \mathcal{R}_r(X)$ and $\mathcal{N}_r(X) \subseteq \mathcal{N}_r(T_1)$. We observe the relation
$S_1 = S_1 (T_1 AS_1)^{-1} T_1 A S_1 = X A S_1$,
which implies that
\begin{equation}\label{3.10.5}
    \mathcal{R}_r(S_1) = \mathcal{R}_r(X A S_1) \subseteq \mathcal{R}_r(X).
\end{equation}
Next, let $u \in \mathcal{N}_r(X)$. This implies that $X u = 0,$
 leading to
$S_1 (T_1 AS_1)^{-1} T_1 u = 0$.
Because $S_1$ is a full column rank matrix, it has a left inverse (see Lemma~\ref{rr1}). Pre-multiplying both sides by the left inverse of $S_1$, we get
$(T_1 AS_1)^{-1} T_1 u = 0$.
Again, by pre-multiplying both sides by $T_1 AS_1$, we obtain $T_1u=0$ which shows that $u \in \mathcal{N}_r(T_1)$. Therefore,
\begin{equation}\label{3.10.6}
  \mathcal{N}_r(X) \subseteq \mathcal{N}_r(T_1).
\end{equation}
Combining equations \eqref{3.10.3}, \eqref{3.10.5}, and \eqref{3.10.6}, we conclude that
$\mathcal{N}_r(X) = \mathcal{N}_r(T_1)$ and $\mathcal{R}_r(X) = \mathcal{R}_r(S_1)$.
Therefore, $A^{(2)}_{r,(S_1,T_1)} = X = A^{(2)}_{r,(W_1,W_1)}.$ This completes the proof.
\end{enumerate}
\end{proof}
The following theorem establishes the explicit representation of $A^{(2)}_{l,\left(W_2, W_2\right)}$, based on the full rank factorization of $W_2$.
\begin{theorem}\label{thm3.11}
Let $A \in \Q_{\nu}^{m \times n}$ and $W_2 \in \Q^{n \times m}$ be a matrix such that $\dim(\mathcal{R}_l(W_2)) = t \leq \nu$ and $\dim(\mathcal{N}_l(W_2)) = n-t$. Assume $W_2 = T_2 S_2$ is a full rank factorization of $W_2$, and define
$X = T_2 (S_2 A T_2)^{-1} S_2$.
If $A$ has a $\{2\}$-inverse $A^{(2)}_{l,(W_2,W_2)}$, then the following statements hold:
\begin{enumerate}[noitemsep]
    \item[$(a)$] $S_2 A T_2$ is an invertible matrix.
    \item[$(b)$] $A^{(2)}_{l,(W_2,W_2)} = X = A^{(2)}_{l,(S_2,T_2)}$.
\end{enumerate}
\end{theorem}
\begin{proof}
\begin{enumerate}
    \item[$(a)$]
According to Theorem \ref{thm3.9}, we have the following relations:
\begin{equation}\label{3.11.1}
    \mathcal{R}_l(W_2)=\mathcal{R}_l(S_2), ~ \mathcal{N}_l(W_2)=\mathcal{N}_l(T_2).
\end{equation}
Now, observe that $\mathcal{R}_l(S_2 A) = \mathcal{R}_l(S_2) A = \mathcal{R}_l(W_2) A$.
Using \cite[Theorem $3.1$]{MR2775924}, we get 
\begin{equation}\label{3.11.2}
 \mathcal{R}_l(W_2) A \oplus \mathcal{N}_l(W_2) = \Q^{1 \times n}.
\end{equation}
Utilizing this, we compute
$$\dim(\mathcal{R}_l(S_2 A)) = \dim(\mathcal{R}_l(W_2) A) 
                            = n- \dim(\mathcal{N}_l(W_2)) = n -(n-t) = t.$$
Therefore, $\rank(S_2 A) = t.$ By a similar argument, we also obtain
$\rank(W_2 A) = t.$ This implies that
$W_2 A = T_2 S_2 A$ is the full rank factorization for $ W_2A$.

Next, suppose $y \in \Q^{1 \times t}$ satisfies $y (S_2 A T_2) = 0$. This implies that $y S_2 A \in \mathcal{N}_l(T_2)$. 
Using equation \eqref{3.11.1}, we obtain $y S_2 A \in \mathcal{N}_l(W_2).$
 In addition, because $y S_2 A \in \mathcal{R}_l(S_2 A) = \mathcal{R}_l(W_2) A$, we have
$y S_2 A \in \mathcal{R}_l(W_2) A \cap \mathcal{N}_l(W_2)$.
Using equation \eqref{3.11.2}, it follows that
$y S_2 A = 0$, which implies $y \in \mathcal{N}_l(S_2 A)$.
Now, observe that
$$\dim(\mathcal{N}_l(S_2 A)) = t - \dim(\mathcal{R}_l(S_2 A)) = t - \rank(S_2 A) = 0,$$
we conclude that $y = 0.$
Hence, for any $y \in \Q^{1 \times t}$, the equation
$y(S_2 A T_2)  = 0$ implies $y = 0$.
Therefore,
$\dim(\mathcal{N}_l(S_2 A T_2)) = 0$,
and consequently, 
$$\rank(S_2 A T_2) = \dim(\mathcal{R}_l(S_2 A T_2)) = t - \dim(\mathcal{N}_l(S_2 A T_2)) = t,$$
implying that $S_2 A T_2$ is a nonsingular matrix; hence, it is invertible. 
\item[$(b)$]
We show that $X$ satisfies $\{2\}$-inverse property:
$$XAX = T_2 (S_2 A T_2)^{-1} S_2 A T_2 (S_2 A T_2)^{-1} S_2 = T_2 (S_2 A T_2)^{-1} S_2 = X,$$
thus, $X \in A\{2\}.$

Applying Definition \ref{def2.2}, we obtain the following results
\begin{equation}\label{3.11.3}
    \mathcal{R}_l(X) = \mathcal{R}_l \left( T_2 (S_2 A T_2)^{-1} S_2 \right) \subseteq \mathcal{R}_l(S_2) ~ \text{and} ~ \mathcal{N}_l(T_2) \subseteq \mathcal{N}_l \left( T_2 (S_2 A T_2)^{-1} S_2 \right) = \mathcal{N}_l(X).
\end{equation}
We will now prove that $\mathcal{R}_l(S_2) \subseteq \mathcal{R}_l(X)$ and $\mathcal{N}_l(X) \subseteq \mathcal{N}_l(T_2)$. We observe the relation
$S_2 = (S_2 A T_2) (S_2 A T_2)^{-1} S_2  = S_2 A X$,
which implies that
\begin{equation}\label{3.11.5}
    \mathcal{R}_l(S_2) = \mathcal{R}_l(S_2 A X) \subseteq \mathcal{R}_l(X).
\end{equation}
Next, let $v \in \mathcal{N}_l(X)$. This implies $v X = 0,$
which leads to
$v T_2 (S_2 A T_2)^{-1} S_2 = 0$.
Since $S_2$ is a full row rank matrix, it has a right inverse (see Lemma~\ref{rr1}). Post-multiplying both sides by the right inverse of $S_2$, we get
$v T_2 (S_2 A T_2)^{-1} = 0.$ Again post-multiplying both sides by $S_2 A T_2$, we obtain $v T_2 =0$ which shows that $v \in \mathcal{N}_l(T_2)$. Therefore,
\begin{equation}\label{3.11.6}
  \mathcal{N}_l(X) \subseteq \mathcal{N}_l(T_2).
\end{equation}
Combining equations \eqref{3.11.3}, \eqref{3.11.5}, and \eqref{3.11.6}, we conclude that
$\mathcal{N}_l(X) = \mathcal{N}_l(T_2)$ and $\mathcal{R}_l(X
  )= \mathcal{R}_l(S_2)$.
Therefore, $A^{(2)}_{l,(S_2,T_2)} = X = A^{(2)}_{l,(W_2,W_2)}$. This completes the proof.
\end{enumerate}
\end{proof}

\section{Computational algorithms and numerical examples}\label{sec4}
In this section, we develop algorithms for computing the generalized inverse of a quaternion matrix under prescribed right and/or left range and null space constraints. Their performance is illustrated through numerical examples. All computations are performed on a system equipped with an Intel Core $i9$-$12900K$ processor running at $3200$ MHz, with $32$ GB of RAM, using MATLAB $R2024b$. 

\subsection{Algorithms}
This subsection is dedicated to designing computational algorithms for computing generalized inverses with predefined subspace constraints, based on the theoretical results established in the previous section. To assess the computational efficiency of the proposed algorithms, we also analyze their arithmetic complexity.
Throughout this subsection, we distinguish between quaternion and complex arithmetic.  It is worth mentioning that  a $\mathbf{q}$-flop denotes one basic arithmetic operation (addition or multiplication) between quaternion numbers, while a $\mathbf{c}$-flop denotes one basic arithmetic operation between complex numbers.
Since a quaternion arithmetic operation can be represented using a fixed number of real or complex arithmetic operations, $\mathbf{q}$-flops and $\mathbf{c}$-flops differ only by constant factors. In this work, we focus on leading-order asymptotic complexity and therefore suppress such constant factors.

Theorem \ref{thm3.4} provides a powerful unified framework for computing generalized inverses, including the Moore-Penrose inverse, the group inverse, the Drazin inverse, and other constrained generalized inverses with prescribed range and null space properties. Specifically, for the given matrices 
$A \in \Q_{\nu}^{m \times n},~S_1 \in \Q^{n \times p},~T_1 \in \Q^{q \times m},
$ the matrix 
$
X = S_1 \left(T_1 A S_1\right)^{(1)} T_1
$
yields different generalized inverses depending on the rank relations among \( T_1 A S_1 \), \( S_1 \), and \( T_1 \). 

To efficiently compute \( X \), we propose the following algorithms:
\begin{itemize}[noitemsep]
    \item \textbf{Algorithm~\ref{alg1}}: Based on direct computation in QTFM. 
    \item \textbf{Algorithm~\ref{alg2}}: Based on the complex representation of quaternion matrices. 
\end{itemize}

\begin{algorithm}[htbp]
{\small{
	\caption{Computation of $X=S_1\left(T_1AS_1\right)^{(1)}T_1$ by using direct computation in QTFM}\label{alg1}
\textbf{Input:} $A \in \Q_{\nu}^{m \times n}$, $S_1 \in \Q^{n \times p}$, and $T_1  \in  \Q^{q \times m}$. \\
	\textbf{Output:} $X=X_1+X_2\j \in \Q^{n \times m}$.
	\begin{enumerate}[noitemsep]
\item Compute the SVD of $W$ using QTFM: $W = U \Sigma V^*$, where $U \in \Q^{q \times q}$ and $V \in \Q^{p \times p}$ are unitary, and 
$\Sigma = \begin{bmatrix} \Sigma_s & 0 \\ 0 & 0 \end{bmatrix} \in \R^{q \times p}$ with $\Sigma_s = \diag(\sigma_1,\ldots,\sigma_s) \in \R^{s \times s}$, $s=\rank(W)$.
        \item Define arbitrary matrices $K \in \Q^{s \times (q-s)}$, $L \in \Q^{(p-s) \times s}$, and $M \in \Q^{(p-s) \times (q-s)}$.
		\item Calculate $W^{(1)}$ as follows:
		$W^{(1)} = V\begin{bmatrix}
			\Sigma_s^{-1} & K \\
			L & M
		\end{bmatrix}U^*$.
		\item Compute $X=S_1W^{(1)}T_1$. 
	\end{enumerate}}}
\end{algorithm}

\begin{algorithm}[h]
{\small{
 \caption{Computation of $X=S_1\left(T_1AS_1\right)^{(1)}T_1$ using the complex representation of quaternion matrices}\label{alg2}
	\textbf{Input:} $A=A_1+A_2\j \in \Q_{\nu}^{m \times n}$, $S_1=S_{11}+S_{12}\j \in \Q^{n \times p}$, and $T_1=T_{11}+T_{12}\j  \in  \Q^{q \times m}$. \\
	\textbf{Output:} $X=X_1+X_2\j \in \Q^{n \times m}$.
	\begin{enumerate}[noitemsep]
   \item \textbf{Compute} $\mathbf{W=W_1+W_2\j=T_1AS_1}$ \textbf{as follows:} \begin{itemize}[noitemsep]
        \item Compute $\left(T_1\right)_r^C$, $A^C$, and $S_1^C$.
        \item Compute $W^C_r$ as: $W^C_r=(T_1)^C_rA^CS_1^C.$
        \item Extract the components of $W$: $
		W_1 = W^C_r(:,1:p), ~ W_2 = W^C_r(:,p+1:end).$
        \end{itemize}
		  \item \textbf{Compute} $\mathbf{\left(W^{(1)}\right)^C}$ \textbf{as follows:}
         \begin{itemize}[noitemsep]
        \item Compute the SVD of $W$ using the complex structure preserving algorithm \cite[Algorithm $3$]{MR4685298}. That is, compute  $W=U \Sigma V^*$, where $U \in \Q^{q \times q}$, $\Sigma=\begin{bmatrix}
			\Sigma_s & 0\\
			0 & 0
		\end{bmatrix} \in \R^{q \times p}$, and $V \in \Q^{p \times p}$. Here, $U$ and $V$ are unitary quaternion matrices, and $\Sigma_{s}=\diag(\sigma_{1},\ldots,\sigma_{s}) \in \R^{s \times s}$ contains the nonzero singular values of $W$. Here, $s=\rank(W)$. 
        \item Define arbitrary matrices $K \in \Q^{s \times (q-s)}$, $L \in \Q^{(p-s) \times s}$, and $M \in \Q^{(p-s) \times (q-s)}$. 
        \item Define $Y=\begin{bmatrix}
			\Sigma_{s}^{-1} & K \\
			L & M
		\end{bmatrix}.$ \\
        \item Compute $V^C$, $Y^C$, and $\left(U^*\right)^C$.
         \item Compute ${\left(W^{(1)}\right)}^C$ as follows:
        $\left(W^{(1)}\right)^C = V^C Y^C (U^*)^C.$
\end{itemize}
         \item \textbf{Compute} $\mathbf{X^C_r}:$ Compute $(S_1)^C_r$ and $(T_1)^C$, and form $X^C_r=\left(S_1\right)^C_r\left(W^{(1)}\right)^C (T_1)^C$.
        \item \textbf{Extract the components of} $\mathbf{X}$:
		$$X_1 = X^C_r(:,1:m), ~ X_2 = X^C_r(:,m+1:end).$$
	\end{enumerate}}}
\end{algorithm}
Both Algorithms \ref{alg1} and \ref{alg2} are based on the same theoretical formulation in Theorem~\ref{thm3.4}, but differ in their computational realization. We now analyze the computational complexity of Algorithms~\ref{alg1} and~\ref{alg2}.
\begin{table}[htbp]
\centering
{\small{
\begin{tabular}{l l|l l}
\hline

\multicolumn{2}{c|}{\textbf{Algorithm 1 ($\mathbf{q}$-flops)}} 
& \multicolumn{2}{c}{\textbf{Algorithm 2 ($\mathbf{c}$-flops)}} \\

\hline

Operation & Cost & Operation & Cost \\

\hline

Compute $W = T_1 A S_1$ 
& $\mathcal{O}(qmn + qnp)$
& Compute $W^C_r = (T_1)^C_r A^C S_1^C$ 
& $\mathcal{O}(qmn + qnp)$ \\

QSVD of $W $ 
& $\mathcal{O}\!\left(pq\min(p,q)\right)$
& SVD of $W$ 
& $\mathcal{O}\!\left(pq\min(p,q)\right)$ \\

Compute $W^{(1)}$ 
& $\mathcal{O}(p^2q + pq^2)$
& Compute $(W^{(1)})^C = V^C Y^C (U^*)^C$ 
& $\mathcal{O}(p^2q + pq^2)$ \\

Compute $X = S_1 W^{(1)} T_1$ 
& $\mathcal{O}(npq + nqm)$
& Compute $X^C_r = (S_1)^C_r (W^{(1)})^C (T_1)^C$ 
& $\mathcal{O}(npq + nqm)$ \\

\hline

\textbf{Total}  
& $\mathcal{O}\!\left(\begin{aligned}
qmn+qnp \\
+\, p^2q + pq^2
\end{aligned}\right)$
& \textbf{Total}  
& $\mathcal{O}\!\left(\begin{aligned}
qmn + qnp \\
+\,  p^2q + pq^2
\end{aligned}\right)$ \\

\hline
\end{tabular}
}}
\caption{Computational complexity of Algorithm~\ref{alg1} and Algorithm~\ref{alg2}.}
\label{table_comparison}
\end{table}

\begin{remark}\label{remt1}
Algorithms~\ref{alg1} and~\ref{alg2} have the same leading-order asymptotic complexity. The difference lies in their computational realization: Algorithm~\ref{alg1} operates directly in the quaternion domain, whereas Algorithm~\ref{alg2} uses a complex representation. 
Although $\mathbf{q}$-flops and $\mathbf{c}$-flops differ only by constant factors, the complex formulation allows the use of optimized BLAS-based routines, resulting in better practical performance for large-scale problems.
\end{remark}

\begin{remark}
The computation of \( X \) in Theorems~\ref{thm3.5} and~\ref{thm3.6} follows the same methodology as in Theorem~\ref{thm3.4}. Consequently, Algorithms~\ref{alg1} and~\ref{alg2} can be directly applied in these settings, and no separate algorithms are required.
\end{remark}

We now discuss the full rank factorization–based method. Using this approach, we develop the following two algorithms for computing the generalized inverse \( A^{(2)}_{r,(W_1, W_1)} \) based on Theorem~\ref{thm3.10}.
\begin{itemize}[noitemsep]
    \item \textbf{Algorithm~\ref{alg3}}: Utilizes direct computation in QTFM.
    \item \textbf{Algorithm~\ref{alg4}}: Employs the complex representation of quaternion matrices.
\end{itemize}
\begin{algorithm}[h]
{\small{
	\caption{Computation of $A^{(2)}_{r,\left(W_1, W_1\right)}$ using direct computation in QTFM}\label{alg3}
\textbf{Input:} $A \in \Q_{\nu}^{m \times n}$ and $W_1 \in \Q^{n \times m}$ such that $\dim\left(\mathcal{R}_r\left(W_1\right)\right) =s \leq \nu$ and $\dim\left(\mathcal{N}_r\left(W_1\right)\right) =m-s$. \\
	\textbf{Output:} $X=X_1+X_2\j=A^{(2)}_{r,\left(W_1, W_1\right)} \in \Q^{n \times m}$.
	\begin{enumerate}[noitemsep]
		\item Compute the full rank factorization of $W_1$ by using direct computation in QTFM \cite[Algorithm $1$]{MR4509094} in order to find matrices $S_1 \in \Q^{n \times s}$ and $T_1 \in \Q^{s \times m}$ such that $W_1=S_1 T_1.$
    \item If $\rank(T_1AS_1)=s$, then compute $X$ as: $ X=S_1(T_1AS_1)^{-1}T_1.$
       \end{enumerate}}}
\end{algorithm}
\begin{algorithm}[H]
{\small{
	\caption{Computation of $A^{(2)}_{r,\left(W_1, W_1\right)}$ using complex representation of quaternion matrices}\label{alg4}
	\textbf{Input:} $A=A_1+A_2\j \in \Q_{\nu}^{m \times n}$ and $W_1=W_{11}+W_{12}\j \in \Q^{n \times m}$ such that $\dim\left(\mathcal{R}_r\left(W_1\right)\right) =s \leq \nu$ and $\dim\left(\mathcal{N}_r\left(W_1\right)\right) =m-s$. \\
	\textbf{Output:} $X=X_1+X_2\j=A^{(2)}_{r,\left(W_1, W_1\right)} \in \Q^{n \times m}$.
	\begin{enumerate}[noitemsep]
		 \item Compute a full rank factorization of $W_1 = W_{11} + W_{12}\j$ using complex structure preserving algorithm \cite[Algorithm~2]{MR4509094}, obtaining $W_1 = S_1 T_1$, where $S_1 =S_{11}+S_{12}\j \in \Q^{n \times s}$ and $T_1=T_{11}+T_{12}\j \in \Q^{s \times m}$.
       \item Compute $T_1^C$, $A^C$, $S_1^C$, and $(S_1)^C_r$.
        \item Define $R=T_1AS_1$. Compute $R^C$ as:
                $ R^C = T_1^CA^CS_1^C.$
      \item If $R^C$ is invertible, then compute $X^C_r$ as:
	       $X^C_r=(S_1)^C_r \left(R^C\right)^{-1}T_1^C.$
      
        \item Extract the components of $X$ as:
       $$X_1 = X^C_r(:,1:m), ~ X_2=X^C_r(:, m+1:end).$$
		\end{enumerate}}}
\end{algorithm}
We now summarize the computational complexity of Algorithms~\ref{alg3} and~\ref{alg4}.
\begin{table}[htbp]
\centering
{\small{
\begin{tabular}{l l|l l}
\hline

\multicolumn{2}{c|}{\textbf{Algorithm 3 ($\mathbf{q}$-flops)}} 
& \multicolumn{2}{c}{\textbf{Algorithm 4 ($\mathbf{c}$-flops)}} \\

\hline

Operation & Cost & Operation & Cost \\

\hline

Full rank factorization of $W_1$ 
& $\mathcal{O}(mns)$
& Full rank factorization of $W_1$ 
& $\mathcal{O}(mns)$ \\

Compute $T_1 A S_1 $ 
& $\mathcal{O}(smn + ns^2)$
& Compute $R^C = T_1^C A^C S_1^C $ 
& $\mathcal{O}(smn + ns^2)$ \\

Compute $(T_1 A S_1)^{-1} $ 
& $\mathcal{O}(s^3)$
& Compute $(R^C)^{-1} $ 
& $\mathcal{O}(s^3)$ \\

Compute $X = S_1 (T_1 A S_1)^{-1} T_1$ 
& $\mathcal{O}(nsm + ns^2)$
& Compute $X^C_r = (S_1)^C_r (R^C)^{-1} T_1^C$ 
& $\mathcal{O}(nsm + ns^2)$ \\

\hline

\textbf{Total}  
& $\mathcal{O}\!\left(\begin{aligned}
mns + ns^2\\
+\,s^3
\end{aligned}\right)$
& \textbf{Total}  
& $\mathcal{O}\!\left(\begin{aligned}
mns + ns^2 \\
+\, s^3
\end{aligned}\right)$ \\

\hline
\end{tabular}
}}
\caption{Computational complexity of Algorithm~\ref{alg3} and Algorithm~\ref{alg4}.}
\label{table_comparison_34}
\end{table}

\begin{remark}\label{remt2}
Algorithms~\ref{alg3} and~\ref{alg4} share identical leading-order asymptotic complexity. However, their practical performance may differ due to the underlying arithmetic and factorization strategy. Algorithm~\ref{alg3} relies on direct quaternion arithmetic through QTFM for computing the full rank factorization and subsequent matrix operations, which may be computationally expensive for large-scale problems. In contrast, Algorithm~\ref{alg4} adopts a complex structure preserving formulation, enabling the use of optimized BLAS-based routines and leading to improved computational efficiency in practice.
\end{remark}

\begin{remark}\label{rem4.2}
The procedure for computing \( A^{(2)}_{l,(W_2, W_2)} \), as described in Theorem~\ref{thm3.11}, is fundamentally the same as that used to compute \( A^{(2)}_{r,(W_1, W_1)} \). Therefore, Algorithms~\ref{alg3} and~\ref{alg4} can  easily be adapted to compute \( A^{(2)}_{l,(W_2, W_2)} \). To avoid redundancy, we did not provide separate algorithms for computing \( A^{(2)}_{l,(W_2, W_2)} \).
\end{remark}

The quaternion Moore--Penrose inverse arises as a special case of the constrained generalized inverses as described in Theorem~\ref{thm3.7}. Consequently, the algorithms developed in Algorithms~\ref{alg1}--\ref{alg4}, along with the associated complexity analysis, apply directly to its computation. For clarity, we summarize the leading-order computational complexities of the
four algorithms for computing the quaternion Moore--Penrose inverse
$A^{\dagger}$ of a matrix $A\in\Q^{m\times n}$ with
$\rank(A)=s\le\min(m,n)$ in Table~\ref{table15}.

\begin{table}[ht]
\centering
\begin{tabular}{@{}ll@{}}
\toprule
Method & Cost (leading order) \\
\midrule
Algorithm~\ref{alg1}
& $\mathcal{O}(m^2n+mn^2)$ $\mathbf{q}$-flops \\

Algorithm~\ref{alg2}
& $\mathcal{O}(m^2n+mn^2)$ $\mathbf{c}$-flops \\

Algorithm~\ref{alg3}
& $\mathcal{O}(mns+ns^2+s^3)$ $\mathbf{q}$-flops \\

Algorithm~\ref{alg4}
& $\mathcal{O}(mns+ns^2+s^3)$ $\mathbf{c}$-flops \\
\bottomrule
\end{tabular}
\caption{Leading-order computational complexity for computing the quaternion
Moore--Penrose inverse.}
\label{table15}
\end{table}

\begin{remark}\label{rem:mp_complexity}
Table~\ref{table15} highlights the distinct computational scaling behaviors of
the proposed algorithms for computing the quaternion Moore--Penrose inverse.
The SVD-based Algorithms~\ref{alg1} and~\ref{alg2} incur a dominant cost of
$\mathcal{O}(m^2n+mn^2)$, which is independent of the numerical rank
$s=\rank(A)$. In contrast, the full rank factorization--based Algorithms~\ref{alg3} and~\ref{alg4}
scale as $\mathcal{O}(mns+ns^2+s^3)$ and are thus computationally preferable
when $A$ is large but of relatively low rank.
Moreover, the complex structure preserving Algorithms~\ref{alg2} and~\ref{alg4}
typically outperform their quaternion-domain counterparts in practice, owing
to the use of optimized complex BLAS routines.
\end{remark}

\subsection{Numerical examples}
In this subsection, we use MATLAB to perform example computations that demonstrate the proposed algorithms. The following example compares the direct QTFM-based method with the complex representation (CR) method for computing the right and left outer generalized inverses with prescribed subspace conditions. 
\begin{example}
We generate random quaternion matrices as follows:
$$A = A_1 + A_2 \j \in \Q^{m \times n},  ~ 
S_1 = S_{11} + S_{12} \j \in \Q^{n \times p},  ~
T_1 = T_{11} + T_{12} \j \in \Q^{q \times m}, ~ 
W_2 = W_{21} + W_{22} \j \in \Q^{n \times m},$$
where, for \(i=1,2\), the complex components \(A_i\), \(S_{1i}\), \(T_{1i}\), and \(W_{2i}\) are generated using \texttt{rand} function.
The matrix dimensions depend on a positive integer \( k \), defined as:
 $m = 3k, n = 2k, p = k, q = k.$ In this example, \( k \) is varied from 5 to $100$ in increments of $5$.  The objective of this example is to compute the right and left generalized inverses of \( A \), denoted as: $A^{(2)}_{r,(S_1, T_1)}, ~  A^{(2)}_{l,(W_2, W_2)}.$ 
 
 For computing the generalized inverses \(A^{(2)}_{r,(S_1, T_1)}\) and \(A^{(2)}_{l,(W_2, W_2)}\), we consider two approaches: direct computation via QTFM (Algorithm~\ref{alg1} and Algorithm \ref{alg3} together with Remark~\ref{rem4.2}) and the complex representation (CR) method (Algorithm~\ref{alg2} and Algorithm \ref{alg4} together with Remark~\ref{rem4.2}).

To evaluate the accuracy and computational efficiency of these methods, let \( X_q^r \) and \( X_c^r \) denote the computed values of \( A^{(2)}_{r,(S_1, T_1)} \). Similarly, let \( X_q^l \) and \( X_c^l \) denote the computed values of \( A^{(2)}_{l,(W_2, W_2)} \). The accuracy of these computations is assessed using the following error metrics:
\begin{eqnarray*}
    \epsilon_q^r = \norm{X_q^r A X_q^r - X_q^r}_F, ~ \epsilon_c^r = \norm{X_c^r A X_c^r - X_c^r}_F, ~ \epsilon_q^l = \norm{X_q^l A X_q^l - X_q^l}_F, ~  \epsilon_c^l = \norm{X_c^l A X_c^l - X_c^l}_F. 
\end{eqnarray*}
Each experiment is run for $50$ trials to ensure the reliability of the results, and the average CPU time is computed. Let \( t_q^r \) and \( t_c^r \) represent the average CPU times required for computing \( X_q^r \) and \( X_c^r \), respectively. Similarly, let \( t_q^l \) and \( t_c^l \) represent the average CPU times  required for computing \( X_q^l \) and \( X_c^l \), respectively. The results are summarized in figures \ref{fig1} and \ref{fig2}.
\end{example}

\begin{figure}[hbt!]
     \centering
     \begin{subfigure}[b]{0.35\textwidth}
         \centering
         \includegraphics[width=\textwidth]{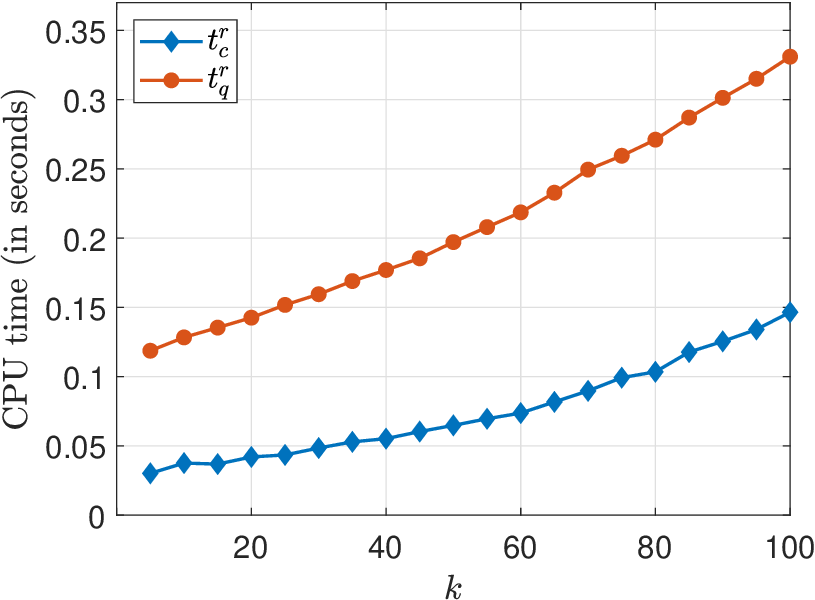}
         \caption{CPU time comparison.}
         \label{fig1a}
     \end{subfigure}
     \begin{subfigure}[b]{0.35\textwidth}
         \centering
         \includegraphics[width=\textwidth]{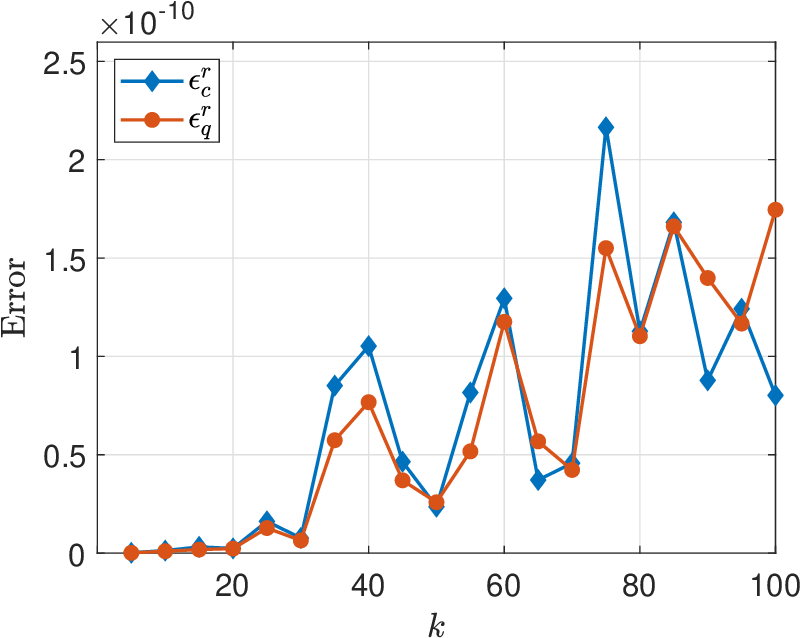}
         \caption{Error comparison.}
         \label{fig1b}
     \end{subfigure}
        \caption{Comparison of CPU time and error for computing  $A^{(2)}_{r,(S_1, T_1)}$ using two methods: direct computation in QTFM and the CR method.}
        \label{fig1}
\end{figure}
For the right generalized inverse \( A^{(2)}_{r,(S_1, T_1)} \), Figure \ref{fig1a} demonstrates that the CR method is consistently faster than the QTFM-based method, which is in line with the discussion in Remark~\ref{remt1}. Figure \ref{fig1b} shows that both methods achieve error values below \( 10^{-10} \), indicating high accuracy.

\begin{figure}[hbt!]
     \centering
     \begin{subfigure}[b]{0.35\textwidth}
         \centering
         \includegraphics[width=\textwidth]{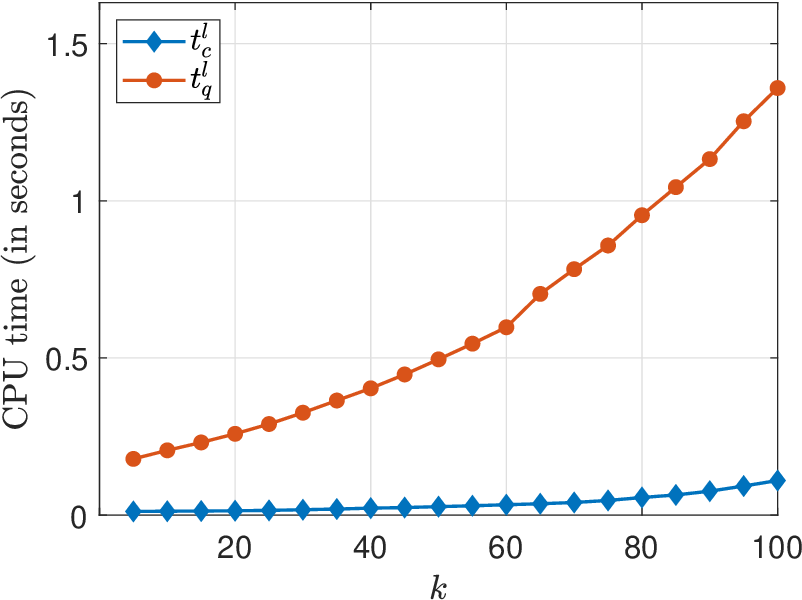}
         \caption{CPU time comparison.}
         \label{fig2a}
     \end{subfigure}
     \begin{subfigure}[b]{0.35\textwidth}
         \centering
         \includegraphics[width=\textwidth]{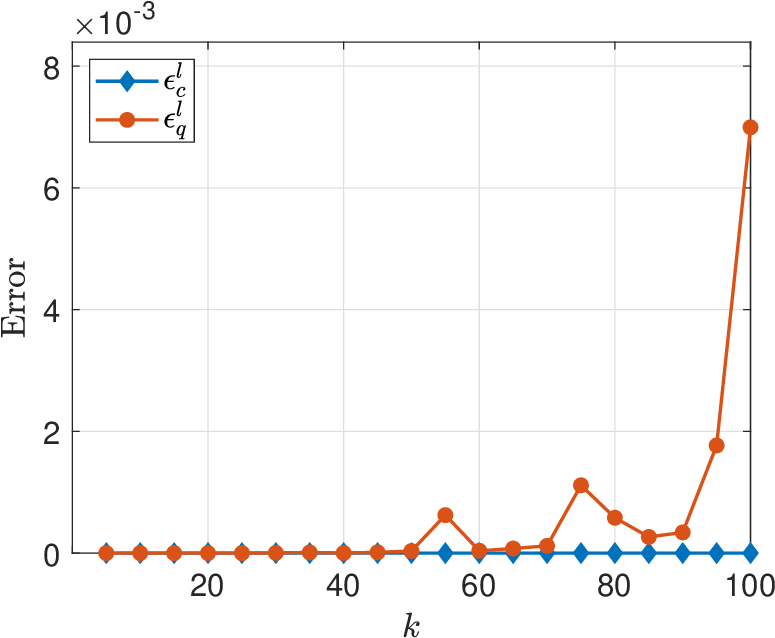}
         \caption{Error comparison.}
         \label{fig2b}
     \end{subfigure}
        \caption{Comparison of CPU time and error for computing  $A^{(2)}_{l,(W_2, W_2)}$ using full rank decomposition by two methods: direct computation in QTFM and the CR method.}
        \label{fig2}
\end{figure}

 For the left generalized inverse \( A^{(2)}_{l,(W_2, W_2)} \), Figure \ref{fig2a} reveals that the CR method outperforms the QTFM-based method in terms of computational speed, which is consistent with the discussion in 
Remark~\ref{remt2}. Figure \ref{fig2b} confirms that the errors remain below \( 10^{-3} \), validating the effectiveness of both methods.

In conclusion, both the QTFM and CR methods produce highly accurate solutions. However, the CR method is significantly faster, making it the preferred choice for large matrices.

\begin{remark}
To the best of our knowledge, the existing literature on generalized inverses over the quaternion skew field, such as \cite{MR2775924}, where determinantal representations and Cramer-type rules were derived, focuses primarily on theoretical characterizations via row and column determinants. In contrast, the present work addresses both theoretical representations and computational techniques for quaternion generalized inverses, including Urquhart-type and full rank decomposition-based formulations, together with practical algorithms. 
\end{remark}

To evaluate the computational efficiency and numerical accuracy of the proposed algorithms for computing the quaternion Moore--Penrose pseudoinverse, we present a numerical example based on randomly generated quaternion matrices. 
\begin{example}
Let $A = A_1 + A_2 \j \in \Q^{m \times n}$,
where the complex matrices $A_1, A_2 \in \mathbb{C}^{m \times n}$ are generated randomly. The matrix dimensions are parameterized by a positive integer $k$ and chosen as $m = 3k$ and $n = 2k$, with $k$ varying from $20$ to $100$ in steps of $20$.

Algorithms~\ref{alg1}–\ref{alg4} are used to compute the Moore--Penrose pseudoinverse, and their performance is evaluated in terms of computational efficiency and numerical accuracy. For each method, let $X$ denote the resulting pseudoinverse. Numerical accuracy is assessed by computing the residuals associated with the four Moore--Penrose conditions.
For each value of $k$, all four algorithms are applied to the same randomly generated quaternion matrix to ensure a consistent comparison. Each experiment is repeated $50$ times, and the reported CPU times are averaged to mitigate the effect of timing variability. The numerical results are summarized in Tables~\ref{tab11}, \ref{tab12}, and \ref{tab13}.
\end{example}

\begin{table}[h]
\centering
\begin{tabular}{c cccc}
\toprule
$k$ & Alg.~\ref{alg1} & Alg.~\ref{alg2} & Alg.~\ref{alg3} & Alg.~\ref{alg4} \\
\midrule
$20$  & $1.538e{-01}$ & $6.154e{-02}$ & $2.512e{-01}$ & $1.148e{-02}$ \\
$40$  & $2.521e{-01}$ & $1.360e{-01}$ & $3.819e{-01}$ & $1.727e{-02}$ \\
$60$  & $4.787e{-01}$ & $3.222e{-01}$ & $5.514e{-01}$ & $2.617e{-02}$ \\
$80$  & $7.889e{-01}$ & $5.562e{-01}$ & $8.831e{-01}$ & $4.630e{-02}$ \\
$100$ & $1.210e{+00}$ & $1.000e{+00}$ & $1.293e{+00}$ & $9.011e{-02}$ \\
\bottomrule
\end{tabular}
\caption{CPU time (seconds) for computing the quaternion pseudoinverse.}
\label{tab11}
\end{table}

\begin{table}[h]
\centering
\begin{tabular}{ccccccccc}
\toprule
 & \multicolumn{4}{c}{$\|AXA - A\|_F$}
 & \multicolumn{4}{c}{$\|XAX-X\|_F$} \\
\cmidrule(lr){2-5}\cmidrule(lr){6-9}
$k$
& Alg.~\ref{alg1} & Alg.~\ref{alg2} & Alg.~\ref{alg3} & Alg.~\ref{alg4}
& Alg.~\ref{alg1} & Alg.~\ref{alg2} & Alg.~\ref{alg3} & Alg.~\ref{alg4} \\
\midrule
$20$  & $2.44e{-11}$ & $2.04e{-11}$ & $2.67e{-09}$ & $1.63e{-10}$ & $1.47e{-12}$ & $2.24e{-12}$ & $4.67e{-10}$ & $5.28e{-12}$ \\
$40$  & $5.38e{-11}$ & $9.68e{-11}$ & $3.94e{-08}$ & $1.87e{-09}$ & $4.78e{-12}$ & $6.55e{-12}$ & $3.71e{-09}$ & $3.51e{-11}$ \\
$60$  & $2.05e{-10}$ & $8.45e{-10}$ & $1.01e{-06}$ & $1.01e{-08}$ & $9.79e{-12}$ & $1.14e{-11}$ & $6.45e{-07}$ & $9.35e{-11}$ \\
$80$  & $3.12e{-10}$ & $3.23e{-10}$ & $5.05e{-06}$ & $3.11e{-08}$ & $1.44e{-11}$ & $9.88e{-12}$ & $2.92e{-07}$ & $1.89e{-10}$ \\
$100$ & $6.17e{-10}$ & $5.89e{-10}$ & $3.17e{-05}$ & $6.99e{-08}$ & $2.21e{-11}$ & $1.33e{-11}$ & $1.31e{-06}$ & $2.83e{-10}$ \\
\bottomrule
\end{tabular}
\caption{Errors associated with the computed quaternion pseudoinverses.}
\label{tab12}
\end{table}

\begin{table}[h]
\centering
\begin{tabular}{ccccccccc}
\toprule
 & \multicolumn{4}{c}{$\|(AX)^* - AX\|_F$}
 & \multicolumn{4}{c}{$\|(XA)^* - XA\|_F$} \\
\cmidrule(lr){2-5}\cmidrule(lr){6-9}
$k$
& Alg.~\ref{alg1} & Alg.~\ref{alg2} & Alg.~\ref{alg3} & Alg.~\ref{alg4}
& Alg.~\ref{alg1} & Alg.~\ref{alg2} & Alg.~\ref{alg3} & Alg.~\ref{alg4} \\
\midrule
$20$  & $3.80e{-12}$ & $5.18e{-12}$ & $1.50e{-09}$ & $4.45e{-11}$ & $3.98e{-12}$ & $5.20e{-12}$ & $1.67e{-09}$ & $2.28e{-11}$ \\
$40$  & $1.50e{-11}$ & $2.20e{-11}$ & $1.49e{-08}$ & $3.91e{-10}$ & $1.53e{-11}$ & $2.13e{-11}$ & $1.71e{-08}$ & $2.22e{-10}$ \\
$60$  & $3.69e{-11}$ & $4.36e{-11}$ & $3.40e{-06}$ & $1.71e{-09}$ & $3.74e{-11}$ & $4.37e{-11}$ & $3.26e{-06}$ & $7.21e{-10}$ \\
$80$  & $6.31e{-11}$ & $4.53e{-11}$ & $1.74e{-06}$ & $4.42e{-09}$ & $6.31e{-11}$ & $4.40e{-11}$ & $1.69e{-06}$ & $2.08e{-09}$ \\
$100$ & $9.96e{-11}$ & $6.40e{-11}$ & $9.01e{-06}$ & $9.45e{-09}$ & $1.02e{-10}$ & $6.50e{-11}$ & $9.18e{-06}$ & $3.55e{-09}$ \\
\bottomrule
\end{tabular}
\caption{Errors associated with the computed quaternion pseudoinverses.}
\label{tab13}
\end{table}

The numerical results in Tables~\ref{tab12} and \ref{tab13} show that all four algorithms compute pseudoinverses that satisfy the Moore--Penrose conditions to high accuracy, with residuals close to machine precision across all tested matrix sizes. Table~\ref{tab11} further indicates that the complex structure preserving methods (Algorithms~\ref{alg2} and~\ref{alg4}) consistently require less CPU time than the direct quaternion-domain approaches as the problem size increases. Moreover, the full rank factorization–based method (Algorithm \ref{alg4}) demonstrate additional efficiency gains for larger matrices. Overall, this example confirms that the proposed algorithms are numerically reliable for computing the quaternion Moore--Penrose pseudoinverse.

\section{Applications}\label{sec7}
In this section, we demonstrate the practical applicability of the proposed algorithms through two distinct real-world scenarios: color image deblurring and three-dimensional signal filtering. These applications highlight the effectiveness of quaternion-based formulations for handling multidimensional data. We first evaluate the efficacy of the proposed methods on a color image deblurring problem. 

\subsection{Color image deblurring}
It is well known that the restoration of degraded color images is a significant challenge in digital image processing. Traditional approaches often process each color channel independently, thereby neglecting the intrinsic correlations between the RGB components. Our quaternion-based formulation addresses this limitation by holistically treating the color image as a quaternion-valued signal. The experimental framework is designed to assess both the computational efficiency and the restoration quality of the recovered images. Color images are typically represented using three channels: red, green, and blue (RGB). In quaternion-based models, these channels are combined into a single quaternion-valued matrix. Specifically, a color image can be expressed as:
$X = X_R \i + X_G \j + X_B \k$, where $X_R, X_G, X_B \in \mathbb{R}^{m \times n}$ denote the individual RGB channels.
This representation allows for holistic processing of color images by preserving the interrelationships between channels. The problem of restoring color images can be formulated as solving the following quaternion linear system:
\begin{equation}\label{eqir}
AX = B,
\end{equation}
where \( A \) is the blurring matrix, \( X \) is the original image to be restored and \( B \) is the observed blurred and noise-free image. The goal is to recover \( X \) from \( B \). To demonstrate the effectiveness of the quaternion-based image restoration method, we consider four original color images with different dimensions: (1) Parrot with an image size $(512 \times 512 )$, (2) Ship with an image size $( 512 \times 768 )$, (3) Scenery with an image size $(2400 \times 2400 )$, and (4) Elephant with an image size $(2482 \times 3374 )$.

These images are blurred using a multichannel blurring matrix represented as \cite{MR4861347}:
$$A = A_1 \i + A_2 \j + A_3 \k \in \Q^{m \times m}.$$
Define \( A_1 = T_0 \otimes T_1 \in \mathbb{R}^{pq \times pq} \), \( A_2 = -0.5 A_1\), and \(A_3 = -0.5 A_1 \). Here, \( T_0=(t^0_{ij}) \in \mathbb{R}^{p \times p} \) and \( T_1=(t^1_{ij}) \in \mathbb{R}^{q \times q} \) are Toeplitz matrices defined as:
\[
t^0_{ij} =
\begin{cases}
    \frac{1}{\sigma \sqrt{2\pi}} \exp{\left(-\frac{(i-j)^2}{2\sigma^2}\right)}, & |i-j| \leq r,\\
    0, & \text{otherwise}.
\end{cases},
~
t^1_{ij} =
\begin{cases}
    \frac{1}{2s-1}, & |i-j| \leq s,\\
    0, & \text{otherwise}.
\end{cases}.
\]
The parameters are set to $\sigma = 3$, $r = 3$, and $s = 3$, with $p=32$ and $q=16$ for the Parrot and Ship images, $p=75$ and $q=32$ for the Scenery image, and $p=34$ and $q=73$ for the Elephant image.


The blurred image \( B \) was generated by computing \( A X \), where \( X \) denotes the original image. 
The true image was reconstructed using the least squares solution \( A^{\dagger} B \), where the Moore--Penrose inverse was computed using Algorithm~\ref{alg1} with \( S_1 = T_1 = A^* \). The visual results for all test images are presented in Figure~\ref{fig:deblurring_results_all}.  

The quality of the restored images is assessed using three metrics:
\begin{itemize}[noitemsep]
    \item Peak Signal-to-Noise Ratio (PSNR): Measures reconstruction quality in decibels (dB).
    \item Structural Similarity Index (SSIM): Assesses the structural similarity between the original and restored images.
    \item Relative Residual (RR): Evaluates the relative error between the original and restored images.
\end{itemize}

The numerical results and visual comparisons of the deblurred images are presented in Table~\ref{tab1} and Figure~\ref{fig:deblurring_results_all}, respectively. These results demonstrate that the quaternion-based method effectively restores color images blurred by the multichannel matrices.

\begin{table}[htbp]
    \centering
    \begin{tabular}{c c c c  c } 
        \toprule
        \textbf{Image} & \textbf{Size}  & \textbf{PSNR} & \textbf{SSIM}  & \textbf{RR} \\
        \midrule
        \text{Parrot} & $512 \times 512$  & $39.10$ & $0.9699$ &  $2.4189 \times 10^{-2} $ \\
        \text{Ship} & $512 \times 768$  & $34.07$ & $0.9378$&  $3.5550 \times 10^{-2} $ \\
        \text{Scenery} & $2400 \times 2400$  & $48.71$ & $0.9909$ &  $7.3889 \times 10^{-3} $ \\
        \text{Elephant} & $2482 \times 3374$  & $43.79$ & $0.9859$ &  $1.7329 \times 10^{-2} $ \\
        \bottomrule
    \end{tabular}
    \caption{Quantitative restoration results for quaternion-based color image deblurring under multichannel blur.}
    \label{tab1}
\end{table}

In \eqref{eqir}, the image restoration is carried out in the quaternion domain, where the color image is treated as a single quaternion-valued signal rather than three independent channels. An equivalent real-valued block formulation is obtained by explicitly separating the RGB components.
Let $B = B_R \i + B_G \j + B_B \k$. By equating the $\i$-, $\j$-, and $\k$-components in \(AX = B\) and using the quaternion multiplication rules, we obtain the coupled real-valued system
$B_R = -A_3 X_G + A_2 X_B$, $B_G = \phantom{-}A_3 X_R - A_1 X_B$, and
$B_B = -A_2 X_R + A_1 X_G$.
This can be written compactly as
\[
\underbrace{
\begin{bmatrix}
0 & -A_3 & \phantom{-}A_2 \\
A_3 & 0 & -A_1 \\
-A_2 & A_1 & 0
\end{bmatrix}
}_{\mathcal{A}_{\mathrm{R}}}
\underbrace{\begin{bmatrix}
X_R \\ X_G \\ X_B
\end{bmatrix}}_{\mathcal{X}_{\mathrm{R}}}
=
\underbrace{\begin{bmatrix}
B_R \\ B_G \\ B_B
\end{bmatrix}}_{\mathcal{B}_{\mathrm{R}}},
\]
so that the restored image in the real-valued framework is given by $\mathcal{X}_{\mathrm{R}}=\mathcal{A}_{\mathrm{R}}^{\dagger}\mathcal{B}_{\mathrm{R}}$.

This derivation shows that the quaternion formulation and the real block formulation are algebraically equivalent representations of the same linear system. The essential difference lies in the modeling structure. In the quaternion framework, the RGB channels are handled as a unified entity, inherently preserving their mutual relationships. In contrast, the real block formulation processes the channels as separate components within a coupled real-valued system. Natural color images exhibit strong statistical correlations among RGB channels. Therefore, an effective restoration method should preserve not only the accuracy but also these inter-channel dependencies.

To quantify the preservation of inter-channel relationships, we compute the $3\times3$ Pearson correlation matrices of the RGB channels over all image pixels. Specifically, each channel is vectorized and the correlation coefficients are computed using Pearson's sample correlation coefficient. We compare these matrices for the original image ($\mathbf{C}_{\mathrm{orig}}$), the quaternion-restored image ($\mathbf{C}_{\mathrm{quat}}$), and the real block restored image ($\mathbf{C}_{\mathrm{real}}$). For the Parrot image, the correlation matrices are

\begin{equation*}
\mathbf{C}_{\mathrm{orig}} =
\begin{bmatrix}
1 & 0.6142 & 0.4227 \\
0.6142 & 1 & 0.6407 \\
0.4227 & 0.6407 & 1
\end{bmatrix}, ~
\mathbf{C}_{\mathrm{quat}} =
\begin{bmatrix}
1 & 0.6129 & 0.4230 \\
0.6129 & 1 & 0.6407 \\
0.4230 & 0.6407 & 1
\end{bmatrix}, ~
\mathbf{C}_{\mathrm{real}} =
\begin{bmatrix}
1 & 0.9010 & 0.8952 \\
0.9010 & 1 & 0.6135 \\
0.8952 & 0.6135 & 1
\end{bmatrix}.
\end{equation*}
For the Ship image:
\begin{equation*}
    \mathbf{C}_{\text{orig}} =
\begin{bmatrix}
1 & 0.9772 & 0.9603 \\
0.9772 & 1 & 0.9911 \\
0.9603 & 0.9911 & 1
\end{bmatrix}, ~
\mathbf{C}_{\text{quat}} =
\begin{bmatrix}
1 & 0.9771 & 0.9603 \\
0.9771 & 1 & 0.9911 \\
0.9603 & 0.9911 & 1
\end{bmatrix}, ~ \mathbf{C}_{\text{real}} =
\begin{bmatrix}
1 & 0.9980 & 0.9979 \\
0.9980 & 1 & 0.9925 \\
0.9979 & 0.9925 & 1
\end{bmatrix}.
\end{equation*}
For the Scenery image: 
\begin{equation*}
\mathbf{C}_{\mathrm{orig}} =
\begin{bmatrix}
1 & 0.8909 & 0.6705 \\
0.8909 & 1 & 0.8988 \\
0.6705 & 0.8988 & 1
\end{bmatrix},~
\mathbf{C}_{\mathrm{quat}} =
\begin{bmatrix}
1 & 0.8909 & 0.6704 \\
0.8909 & 1 & 0.8987 \\
0.6704 & 0.8987 & 1
\end{bmatrix}, ~
\mathbf{C}_{\mathrm{real}} =
\begin{bmatrix}
1 & 0.9845 & 0.9743 \\
0.9845 & 1 & 0.9199 \\
0.9743 & 0.9199 & 1
\end{bmatrix}.
\end{equation*}
For the Elephant image:
\begin{equation*}
   \mathbf{C}_{\mathrm{orig}} =
\begin{bmatrix}
1 & 0.9722 & 0.9137 \\
0.9722 & 1 & 0.9775 \\
0.9137 & 0.9775 & 1
\end{bmatrix}, ~
\mathbf{C}_{\mathrm{quat}} =
\begin{bmatrix}
1 & 0.9721 & 0.9135 \\
0.9721 & 1 & 0.9773 \\
0.9135 & 0.9773 & 1
\end{bmatrix}, ~
\mathbf{C}_{\mathrm{real}} =
\begin{bmatrix}
1 & 0.9949 & 0.9947 \\
0.9949 & 1 & 0.9793 \\
0.9947 & 0.9793 & 1
\end{bmatrix}.
\end{equation*}

The visual comparisons in Figure~\ref{fig:deblurring_results_all} support these observations. For clarity of presentation, we display representative results for the Parrot and Ship images, while the quantitative analysis is conducted on all test images.  The quaternion-restored images preserve color consistency and overall visual coherence. In contrast, the real block restorations exhibit noticeable color shifts and artificial channel-mixing effects. This visual behavior is consistent with the correlation analysis.
For all four images, $\mathbf{C}_{\mathrm{quat}}$ is nearly identical to $\mathbf{C}_{\mathrm{orig}}$, indicating that the original inter-channel relationships are effectively preserved under the quaternion formulation. 
In contrast, $\mathbf{C}_{\mathrm{real}}$ shows inflated off-diagonal entries in several cases, reflecting spurious cross-channel coupling that aligns with the observed color distortions.

Overall, the experiments demonstrate that the quaternion formulation achieves high numerical restoration accuracy while preserving the intrinsic RGB correlation structure. This makes the quaternion framework a more suitable and structurally consistent model for color image deblurring than the corresponding real block representation.

\begin{figure}[H]
\centering
\begin{tabular}{cccc}

\includegraphics[width=0.18\textwidth]{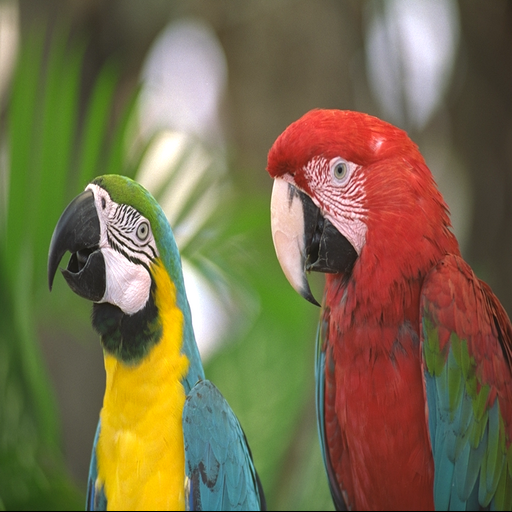} &
\includegraphics[width=0.18\textwidth]{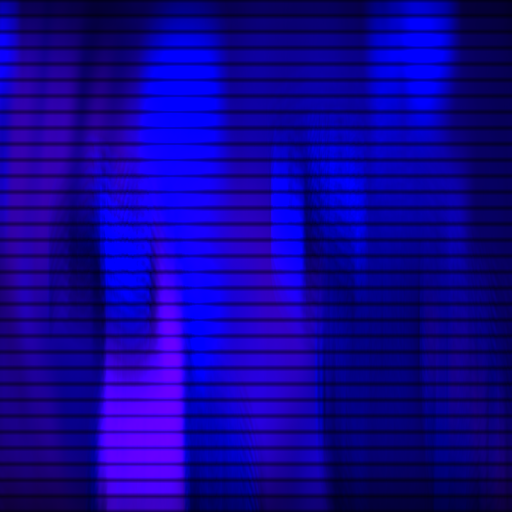} &
\includegraphics[width=0.18\textwidth]{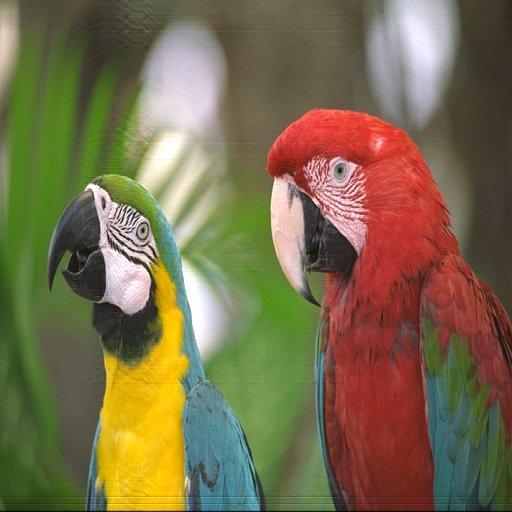} &
\includegraphics[width=0.18\textwidth]{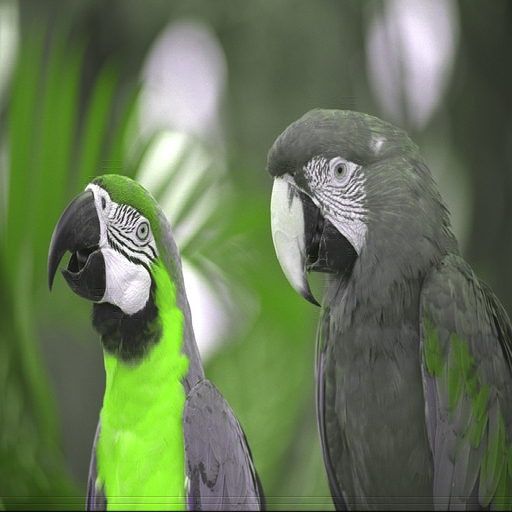} \\[0.2cm]

\includegraphics[width=0.18\textwidth]{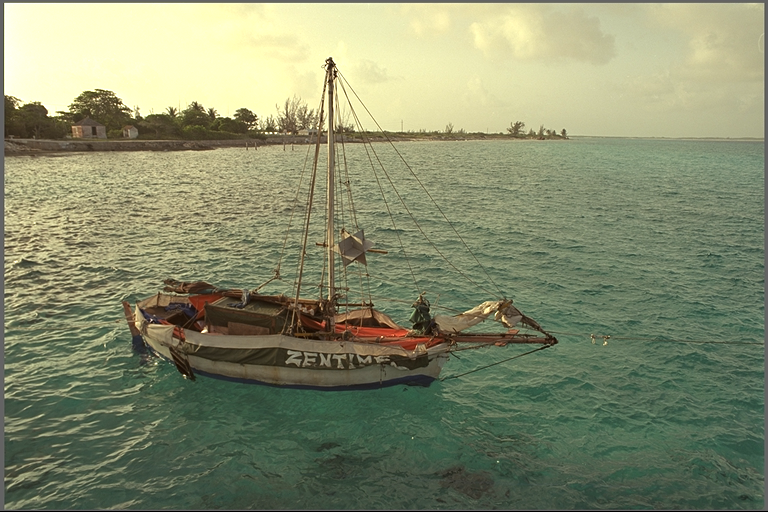} &
\includegraphics[width=0.18\textwidth]{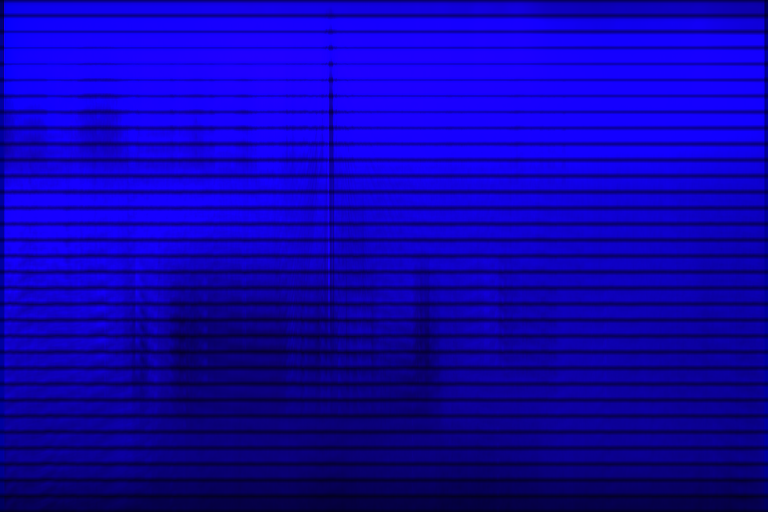} &
\includegraphics[width=0.18\textwidth]{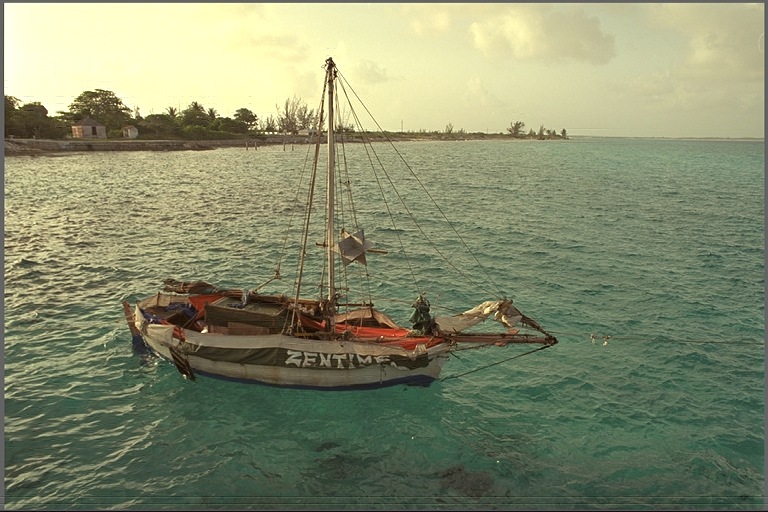} &
\includegraphics[width=0.18\textwidth]{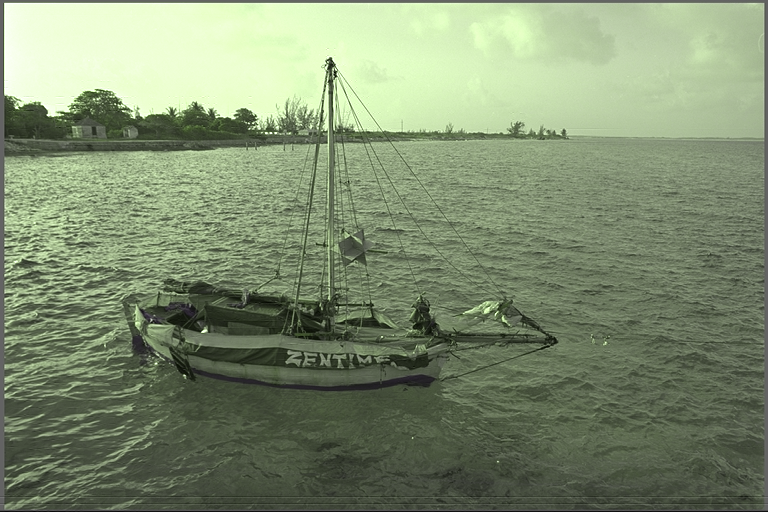} \\[0.2cm]



Original image &
Blurred image &
\shortstack{Deblurred image\\(Quaternion)} &
\shortstack{Deblurred image\\(Real)}

\end{tabular}

\caption{Deblurring results for four test images (Parrot and Ship). Each row corresponds to a different image, while the columns represent the original image, the blurred observation, the restored image using the quaternion formulation, and the restored image using the real-valued formulation.}
\label{fig:deblurring_results_all}
\end{figure}

\subsection{Three-dimensional signal filtering}
The proposed quaternion-based algorithm addresses key challenges in real-world applications, particularly in handling chaotic signals such as those generated by the Lorenz system, where high dimensionality and nonlinear dynamics require robust filtering techniques to recover and interpret meaningful signal patterns.

We consider the following example to showcase the effectiveness of quaternion-valued filtering in recovering signals distorted by delay and noise, thus illustrating the versatility of the proposed approach in processing multidimensional time series data. The Lorenz system is governed by the following set of coupled ordinary differential equations \cite{Strogatz2015}:
\begin{equation}
\frac{dx}{dt} = \alpha (y - x), ~
\frac{dy}{dt} = x(\rho - z) - y, ~
\frac{dz}{dt} = xy - \beta z.
\label{eq10.1}
\end{equation}
where \(\alpha\), \(\beta\), and \(\rho\) are strictly positive parameters. For classical chaotic behavior, the standard values \(\alpha = 10\), \(\beta = \tfrac{8}{3}\), and \(\rho = 28\) are employed. The system is simulated over the interval \(t \in [0,T]\) with sampling step \(\Delta t\), starting from the initial condition \((x(0),y(0),z(0)) = (1,1,1)\). We define a three-dimensional quaternion-valued signal as
$\mathbf{a}(t) = a_r(t) \i + a_g(t) \j + a_b(t) \k,$ where $a_r(t), a_g(t), a_b(t) \in \mathbb{R}$ are the real functions. Suppose that the solutions to \eqref{eq10.1} are $d_r(t)$, $d_g(t)$, and $d_b(t)$. Then, the desired output signal is given by
\[
\mathbf{d}(t) = d_r(t) \i + d_g(t) \j + d_b(t) \k.
\]
To generate the input signal, we introduce a delay and a noise component:
\[
\mathbf{c}(t) = d_r(t - 1) \i + d_g(t - 1) \j + d_b(t - 1) \k + \mathbf{n}(t),
\]
where $\mathbf{n}(t)$ represents purely imaginary quaternion-valued random noise. We aim to determine a set of quaternion-valued filter coefficients $\{\mathbf{f}(k)\}_{k=0}^{n}$, where $\mathbf{f}(k)=f(k)_0+f(k)_r \i + f(k)_g \j + f(k)_b \k$   such that the convolution of the input signal $\mathbf{c}(t)$ with the filter approximates the target output signal $\mathbf{d}(t)$:
\begin{equation}
\mathbf{d}(t) = \sum_{k=0}^{n} \mathbf{c}(t - k) * \mathbf{f}(k).
\label{eq10.2}
\end{equation}
To express the above equation \label{eq10.2} in matrix form, we define the quaternion-valued data matrix, filter vector, and output vector as follows:
\begin{eqnarray*}
\mathbf{C} = 
\begin{bmatrix}
\mathbf{c}(t) & \mathbf{c}(t-1) & \cdots & \mathbf{c}(t-n) \\
\mathbf{c}(t+1) & \mathbf{c}(t) & \cdots & \mathbf{c}(t - n + 1) \\
\vdots & \vdots & \ddots & \vdots \\
\mathbf{c}(t + n) & \mathbf{c}(t + n - 1) & \cdots & \mathbf{c}(t )
\end{bmatrix}, 
~~~\mathbf{f} = \begin{bmatrix}
\mathbf{f}(0)\\\mathbf{f}(1)\\ \vdots \\\mathbf{f}(n)
\end{bmatrix},
~~~\mathbf{d} = \begin{bmatrix}
\mathbf{d}(t)\\ \mathbf{d}(t + 1) \\\vdots \\ \mathbf{d}(t + n)
\end{bmatrix}. 
\end{eqnarray*}
We now express the quaternion-valued linear system as 
\begin{equation}
\mathbf{C}  \mathbf{f} = \mathbf{d},
\label{eq10.3}
\end{equation}
where \(\mathbf{C} \in \Q^{(n+1) \times (n+1)}\) is the data matrix, \(\mathbf{f} \in \Q^{(n+1) \times 1}\) is the filter coefficient vector, and \(\mathbf{d} \in \Q^{(n+1) \times 1}\) is the target signal vector. To obtain the optimal filter coefficient \(\mathbf{f}\), we compute the quaternion pseudoinverse solution \(\mathbf{f} = \mathbf{C}^{\dagger}\mathbf{d}\), where \(\mathbf{C}^{\dagger}\) is evaluated using Algorithm~\ref{alg2}. Once the filter is applied, the estimated output signal is given by \(\hat{\mathbf{d}} = \mathbf{C}\mathbf{f}\). The estimation error metric is defined as the relative reconstruction error
\[
\mathbf{e}
=
\frac{\|\hat{\mathbf{d}} - \mathbf{d}\|}{\|\mathbf{d}\|},
\]
that quantifies the discrepancy between the predicted and actual target signals. This quantity serves as a performance metric for evaluating the effectiveness of the filtering process. 

The visual results reported in Figure~\ref{fig:three_combined} correspond to the numerical experiment with parameters $T=50$ and $\Delta t=0.06$. For this case, the resulting quaternion linear system has size
$\mathbf{C}\in\Q^{417\times417}$, and the proposed method yields a relative error of $\mathbf{e}=8.9978\times10^{-11}$. 

\begin{figure}[hbt!]
    \centering
    
    \begin{subfigure}[b]{0.32\textwidth}
        \centering
        \includegraphics[width=\textwidth]{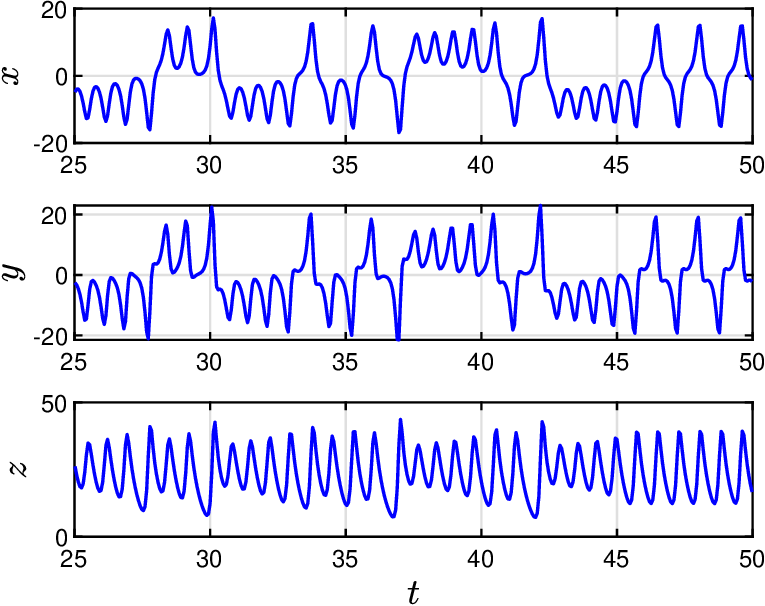}
        \caption{}
        \label{fig11b}
    \end{subfigure}
    \hfill
    \begin{subfigure}[b]{0.32\textwidth}
        \centering
        \includegraphics[width=\textwidth]{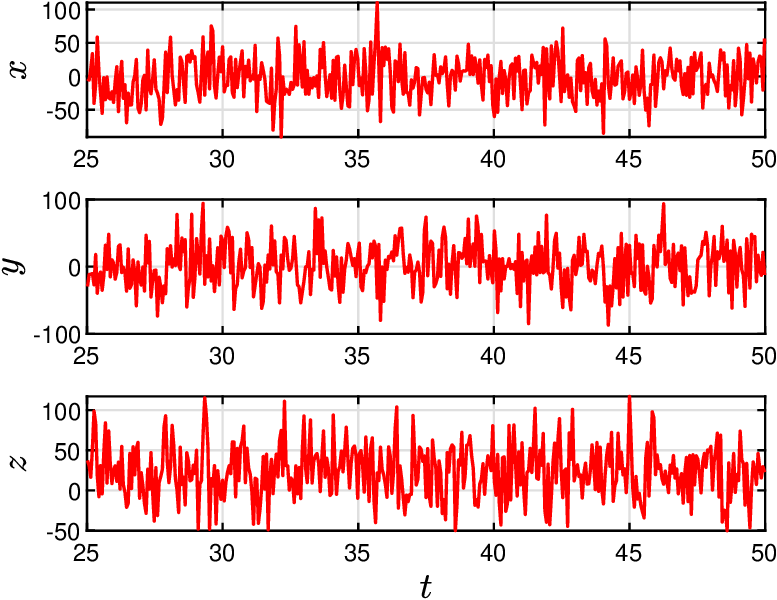}
        \caption{}
        \label{fig12a}
    \end{subfigure}
    \hfill
    \begin{subfigure}[b]{0.32\textwidth}
        \centering
        \includegraphics[width=\textwidth]{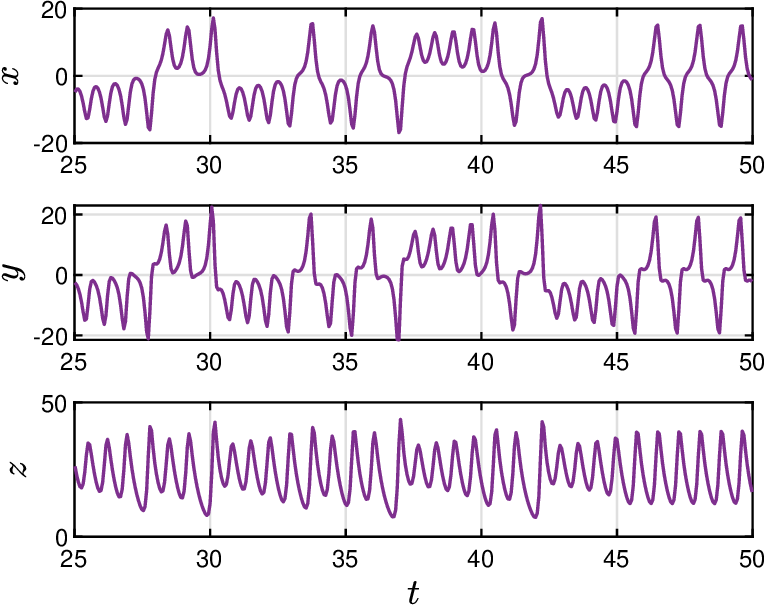}
        \caption{}
        \label{fig12b}
    \end{subfigure}

    \caption{(a) Target signal components, 
    (b) Noisy and delayed signal components, 
    (c) Estimated signal components after filtering.}
    
    \label{fig:three_combined}
\end{figure}

Figure~\ref{fig:three_combined}(a) shows the clean target signal components obtained from the solution of the system. Figure~\ref{fig:three_combined}(b) presents the input signal after introducing a unit time delay and additive Gaussian noise, which distorts the original dynamics and mimics practical signal degradation scenarios. Figure~\ref{fig:three_combined}(c) shows the estimated signal components obtained by applying the learned quaternion filter to the noisy input. 
The close agreement between the estimated and target signals visually confirms the effectiveness of the proposed filtering framework. These visualizations support the numerical result of the low estimation error, \(\mathbf{e} = 8.9978\times10^{-11}\), demonstrating the high fidelity of the reconstructed signal.

In addition to the above experiment, we also tested the method for a larger time horizon. 
For $T = 250$ with sampling step $\Delta t = 0.03$, the resulting quaternion data matrix $\mathbf{C}$ has size $4167 \times 4167$. 
The proposed method yields a relative reconstruction error of 
$3.6663 \times 10^{-7}$, demonstrating that the algorithm maintains high accuracy even for substantially larger systems.

\section{Conclusions} \label{sec5}
 This study investigated the theoretical and computational aspects of generalized inverses for quaternion matrices, with a particular emphasis on outer inverses and $\{1,2\}$-inverses with a prescribed range and/or null space constraints. We provided a detailed characterization of the left and right ranges and null spaces of quaternion matrices, addressing the challenges arising from the non-commutative nature of quaternions. Explicit representations for these generalized inverses were derived, along with full rank decomposition-based representations.  Furthermore, a unified theoretical framework was established, showing that outer inverses with prescribed subspace constraints can recover several classical generalized inverses such as the Moore Penrose inverse, the group inverse, and the Drazin inverse as special cases.

To bridge the gap between theory and practice, we proposed two efficient algorithms for computing these inverses: a direct algorithm using QTFM and a complex structure preserving algorithm. These algorithms were validated using extensive numerical experiments that demonstrated their accuracy and efficiency. The practical utility of the proposed methods was demonstrated through applications in color image deblurring and three-dimensional signal filtering, highlighting their effectiveness in solving real-world problems.

Future work could focus on developing iterative methods for efficiently computing quaternion generalized inverses, particularly for large-scale problems. Furthermore, exploring their applications in other areas of signal processing, machine learning, and beyond remains a promising direction for further research. This study not only advances the theoretical understanding of quaternion generalized inverses but also provides practical tools for their computation and application, paving the way for future innovations in this area.

\section*{Declarations}
\subsection*{Ethical approval}
Not applicable
\subsection*{Availability of supporting data}
The authors confirm that the data supporting the findings of this study are available within the article. 
\subsection*{Competing interests}
The author declares no competing interests.
\subsection*{Funding}
Ratikanta Behera is grateful for the support of the Anusandhan National Research Foundation (ANRF), Govt. of India, under Grant No. EEQ/2022/001065. 
 
\section*{ORCID}
Ratikanta Behera\orcidA \href{https://orcid.org/0000-0002-6237-5700}{ \hspace{2mm}\textcolor{lightblue}{ https://orcid.org/0000-0002-6237-5700}}\\
Neha Bhadala \orcidC \href{https://orcid.org/0009-0001-9249-0611}{ \hspace{2mm}\textcolor{lightblue}{https://orcid.org/0009-0001-9249-0611}} \\

\bibliographystyle{abbrv}
\bibliography{References}

@article {MR2775924,
    AUTHOR = {Song, Guang-Jing and Wang, Qing-Wen and Chang, Hai-Xia},
     TITLE = {Cramer rule for the unique solution of restricted matrix
              equations over the quaternion skew field},
   JOURNAL = {Comput. Math. Appl.},
  FJOURNAL = {Computers \& Mathematics with Applications. An International
              Journal},
    VOLUME = {61},
      YEAR = {2011},
    NUMBER = {6},
     PAGES = {1576--1589},
      ISSN = {0898-1221,1873-7668},
   MRCLASS = {15A24 (15A15 15B33)},
  MRNUMBER = {2775924},
MRREVIEWER = {Enric\ Ventura Capell},
       DOI = {10.1016/j.camwa.2011.01.026},
       URL = {https://doi.org/10.1016/j.camwa.2011.01.026},
}

@book {MR1333599,
    AUTHOR = {Adler, Stephen L.},
     TITLE = {Quaternionic quantum mechanics and quantum fields},
    SERIES = {International Series of Monographs on Physics},
    VOLUME = {88},
 PUBLISHER = {The Clarendon Press, Oxford University Press, New York},
      YEAR = {1995},
     PAGES = {xii+586},
      ISBN = {0-19-506643-X},
   MRCLASS = {81-02 (81Qxx 81Txx 81Uxx)},
  MRNUMBER = {1333599},
MRREVIEWER = {H.\ Hogreve},
}

@book {MR3308953,
    AUTHOR = {Ell, Todd A. and Le Bihan, Nicolas and Sangwine, Stephen J.},
     TITLE = {Quaternion {F}ourier transforms for signal and image
              processing},
    SERIES = {Focus Series in Digital Signal and Image Processing},
 PUBLISHER = {John Wiley \& Sons, Inc., Hoboken, NJ; ISTE, London},
      YEAR = {2014},
     PAGES = {xxii+127},
      ISBN = {978-1-84821-478-1},
   MRCLASS = {94-02 (43A30 94-01 94A08 94A12)},
  MRNUMBER = {3308953},
       DOI = {10.1002/9781118930908},
       URL = {https://doi.org/10.1002/9781118930908},
}

@article {MR4572388,
    AUTHOR = {Morais, J. and Ferreira, M.},
     TITLE = {Hyperbolic linear canonical transforms of quaternion signals
              and uncertainty},
   JOURNAL = {Appl. Math. Comput.},
  FJOURNAL = {Applied Mathematics and Computation},
    VOLUME = {450},
      YEAR = {2023},
     PAGES = {Paper No. 127971, 24},
      ISSN = {0096-3003,1873-5649},
   MRCLASS = {44A05},
  MRNUMBER = {4572388},
       DOI = {10.1016/j.amc.2023.127971},
       URL = {https://doi.org/10.1016/j.amc.2023.127971},
}

@article {MR4867322,
    AUTHOR = {Huang, Baohua and Jia, Zhigang and Li, Wen},
     TITLE = {Conjugate gradient normal residual method for solving
              quaternion {T}oeplitz-type linear systems with application to
              color image processing},
   JOURNAL = {J. Sci. Comput.},
  FJOURNAL = {Journal of Scientific Computing},
    VOLUME = {103},
      YEAR = {2025},
    NUMBER = {1},
     PAGES = {Paper No. 12, 33},
      ISSN = {0885-7474,1573-7691},
   MRCLASS = {65F10 (15B33 94A08)},
  MRNUMBER = {4867322},
       DOI = {10.1007/s10915-025-02826-z},
       URL = {https://doi.org/10.1007/s10915-025-02826-z},
}

@article {MR4861347,
    AUTHOR = {Li, Tao and Wang, Qing-Wen and Zhang, Xin-Fang},
     TITLE = {Gl-{QFOM} and {G}l-{QGMRES}: two efficient algorithms for
              quaternion linear systems with multiple right-hand sides},
   JOURNAL = {Numer. Linear Algebra Appl.},
  FJOURNAL = {Numerical Linear Algebra with Applications},
    VOLUME = {32},
      YEAR = {2025},
    NUMBER = {1},
     PAGES = {Paper No. e70008, 15},
      ISSN = {1070-5325,1099-1506},
   MRCLASS = {65F10 (15B33 65F55)},
  MRNUMBER = {4861347},
MRREVIEWER = {Zhaolu\ Tian},
       DOI = {10.1002/nla.70008},
       URL = {https://doi.org/10.1002/nla.70008},
}

@book {Hamilton2009,
    AUTHOR = {Hamilton, William Rowan},
     TITLE = {Elements of quaternions. {P}art 2},
    EDITOR = {Hamilton, William Edwin},
     PUBLISHER = {Cambridge University Press, Cambridge},
      YEAR = {2009},
     PAGES = {i--ii and 301--762},
      ISBN = {978-1-108-00899-0},
   MRCLASS = {01A75},
  MRNUMBER = {2884277},
       DOI = {10.1017/CBO9780511707162.017},
       URL = {https://doi.org/10.1017/CBO9780511707162.017},
}

@article {MR1421264,
    AUTHOR = {Zhang, Fuzhen},
     TITLE = {Quaternions and matrices of quaternions},
   JOURNAL = {Linear Algebra Appl.},
  FJOURNAL = {Linear Algebra and its Applications},
    VOLUME = {251},
      YEAR = {1997},
     PAGES = {21--57},
      ISSN = {0024-3795,1873-1856},
   MRCLASS = {15A33 (16K20)},
  MRNUMBER = {1421264},
MRREVIEWER = {Walter\ S.\ Sizer},
       DOI = {10.1016/0024-3795(95)00543-9},
       URL = {https://doi.org/10.1016/0024-3795(95)00543-9},
}

@book {MR3241695,
    AUTHOR = {Rodman, Leiba},
     TITLE = {Topics in quaternion linear algebra},
    SERIES = {Princeton Series in Applied Mathematics},
 PUBLISHER = {Princeton University Press, Princeton, NJ},
      YEAR = {2014},
     PAGES = {xii+363},
      ISBN = {978-0-691-16185-3},
   MRCLASS = {15-01 (15B33)},
  MRNUMBER = {3241695},
MRREVIEWER = {Gisele\ C.\ Ducati},
       DOI = {10.1515/9781400852741},
       URL = {https://doi.org/10.1515/9781400852741},
}

@article {MR4421890,
    AUTHOR = {Qi, Liqun and Luo, Ziyan and Wang, Qing-Wen and Zhang,
              Xinzhen},
     TITLE = {Quaternion matrix optimization: motivation and analysis},
   JOURNAL = {J. Optim. Theory Appl.},
  FJOURNAL = {Journal of Optimization Theory and Applications},
    VOLUME = {193},
      YEAR = {2022},
    NUMBER = {1-3},
     PAGES = {621--648},
      ISSN = {0022-3239,1573-2878},
   MRCLASS = {90C30 (15B33 68U10)},
  MRNUMBER = {4421890},
       DOI = {10.1007/s10957-021-01906-y},
       URL = {https://doi.org/10.1007/s10957-021-01906-y},
}

@article {MR2672082,
    AUTHOR = {Jiang, Tongsong and Chen, Li},
     TITLE = {An algebraic method for {S}chr\"odinger equations in
              quaternionic quantum mechanics},
   JOURNAL = {Comput. Phys. Comm.},
  FJOURNAL = {Computer Physics Communications. An International Journal and
              Program Library for Computational Physics and Physical
              Chemistry},
    VOLUME = {178},
      YEAR = {2008},
    NUMBER = {11},
     PAGES = {795--799},
      ISSN = {0010-4655},
   MRCLASS = {81Q65 (81Q05)},
  MRNUMBER = {2672082},
       DOI = {10.1016/j.cpc.2008.01.038},
       URL = {https://doi.org/10.1016/j.cpc.2008.01.038},
}

@article{ChenJiaPeng2021,
  title={A new structure-preserving quaternion QR decomposition method for color image blind watermarking},
  author={Chen, Yong and Jia, Zhi-Gang and Peng, Yan and Peng, Ya-Xin and Zhang, Dan},
  journal={Signal Processing},
  volume={185},
  pages={108088},
  year={2021},
  publisher={Elsevier}
}

@article {MR4682073,
    AUTHOR = {Jiang, Tongsong and Guo, Zhenwei and Zhang, Dong and Vasil'ev,
              V. I.},
     TITLE = {A fast algorithm for the {S}chr\"odinger equation in
              quaternionic quantum mechanics},
   JOURNAL = {Appl. Math. Lett.},
  FJOURNAL = {Applied Mathematics Letters. An International Journal of Rapid
              Publication},
    VOLUME = {150},
      YEAR = {2024},
     PAGES = {Paper No. 108975, 6},
      ISSN = {0893-9659,1873-5452},
   MRCLASS = {81-08 (35J10 35P15 81Q65)},
  MRNUMBER = {4682073},
MRREVIEWER = {Gisele\ C.\ Ducati},
       DOI = {10.1016/j.aml.2023.108975},
       URL = {https://doi.org/10.1016/j.aml.2023.108975},
}

@book {Vince2021,
    AUTHOR = {Vince, John},
     TITLE = {Quaternions for computer graphics},
   EDITION = {Second},
 PUBLISHER = {Springer, London},
      YEAR = {[2021] \copyright 2021},
     PAGES = {xv+181},
      ISBN = {978-1-4471-7508-7; 978-1-4471-7509-4},
   MRCLASS = {68-01 (68U05)},
  MRNUMBER = {4397630},
       DOI = {10.1007/978-1-4471-7509-4},
       URL = {https://doi.org/10.1007/978-1-4471-7509-4},
}

@article {MR4509094,
    AUTHOR = {Wang, Gang and Zhang, Dong and Vasiliev, Vasily. I. and Jiang,
              Tongsong},
     TITLE = {A complex structure-preserving algorithm for the full rank
              decomposition of quaternion matrices and its applications},
   JOURNAL = {Numer. Algorithms},
  FJOURNAL = {Numerical Algorithms},
    VOLUME = {91},
      YEAR = {2022},
    NUMBER = {4},
     PAGES = {1461--1481},
      ISSN = {1017-1398,1572-9265},
   MRCLASS = {65F10 (15B33)},
  MRNUMBER = {4509094},
       DOI = {10.1007/s11075-022-01310-1},
       URL = {https://doi.org/10.1007/s11075-022-01310-1},
}

@article {MR4685298,
    AUTHOR = {Zhang, Dong and Jiang, Tongsong and Jiang, Chuan and Wang,
              Gang},
     TITLE = {A complex structure-preserving algorithm for computing the
              singular value decomposition of a quaternion matrix and its
              applications},
   JOURNAL = {Numer. Algorithms},
  FJOURNAL = {Numerical Algorithms},
    VOLUME = {95},
      YEAR = {2024},
    NUMBER = {1},
     PAGES = {267--283},
      ISSN = {1017-1398,1572-9265},
   MRCLASS = {65F55 (15B33)},
  MRNUMBER = {4685298},
MRREVIEWER = {Minghui\ Wang},
       DOI = {10.1007/s11075-023-01571-4},
       URL = {https://doi.org/10.1007/s11075-023-01571-4},
}

@article {MR3735827,
    AUTHOR = {Li, Ying and Wei, Musheng and Zhang, Fengxia and Zhao, Jianli},
     TITLE = {A real structure-preserving method for the quaternion {LU}
              decomposition, revisited},
   JOURNAL = {Calcolo},
  FJOURNAL = {Calcolo. A Quarterly on Numerical Analysis and Theory of
              Computation},
    VOLUME = {54},
      YEAR = {2017},
    NUMBER = {4},
     PAGES = {1553--1563},
      ISSN = {0008-0624,1126-5434},
   MRCLASS = {15A23 (15B33 65F05)},
  MRNUMBER = {3735827},
MRREVIEWER = {K.\ B.\ Datta},
       DOI = {10.1007/s10092-017-0241-4},
       URL = {https://doi.org/10.1007/s10092-017-0241-4},
}

@article {MR2368224,
    AUTHOR = {Sheng, Xingping and Chen, Guoliang},
     TITLE = {Full-rank representation of generalized inverse
              {$A^{(2)}_{T,S}$} and its application},
   JOURNAL = {Comput. Math. Appl.},
  FJOURNAL = {Computers \& Mathematics with Applications. An International
              Journal},
    VOLUME = {54},
      YEAR = {2007},
    NUMBER = {11-12},
     PAGES = {1422--1430},
      ISSN = {0898-1221,1873-7668},
   MRCLASS = {15A09},
  MRNUMBER = {2368224},
MRREVIEWER = {Predrag\ S.\ Stanimirovi\'c},
       DOI = {10.1016/j.camwa.2007.05.011},
       URL = {https://doi.org/10.1016/j.camwa.2007.05.011},
}

@article {MR2451534,
    AUTHOR = {Stanimirovi\'c, Predrag S. and Cvetkovi\'c-Ili\'c, Dragana S.},
     TITLE = {Successive matrix squaring algorithm for computing outer
              inverses},
   JOURNAL = {Appl. Math. Comput.},
  FJOURNAL = {Applied Mathematics and Computation},
    VOLUME = {203},
      YEAR = {2008},
    NUMBER = {1},
     PAGES = {19--29},
      ISSN = {0096-3003,1873-5649},
   MRCLASS = {15A09},
  MRNUMBER = {2451534},
       DOI = {10.1016/j.amc.2008.04.037},
       URL = {https://doi.org/10.1016/j.amc.2008.04.037},
}

@article {MR3162349,
    AUTHOR = {Stanimirovi\'c, Predrag S. and Soleymani, F.},
     TITLE = {A class of numerical algorithms for computing outer inverses},
   JOURNAL = {J. Comput. Appl. Math.},
  FJOURNAL = {Journal of Computational and Applied Mathematics},
    VOLUME = {263},
      YEAR = {2014},
     PAGES = {236--245},
      ISSN = {0377-0427,1879-1778},
   MRCLASS = {65F05 (15A09)},
  MRNUMBER = {3162349},
MRREVIEWER = {Ioana\ Chiorean},
       DOI = {10.1016/j.cam.2013.12.033},
       URL = {https://doi.org/10.1016/j.cam.2013.12.033},
}

@article {MR1645022,
    AUTHOR = {Wei, Yimin},
     TITLE = {A characterization and representation of the generalized
              inverse {$A^{(2)}_{T,S}$} and its applications},
   JOURNAL = {Linear Algebra Appl.},
  FJOURNAL = {Linear Algebra and its Applications},
    VOLUME = {280},
      YEAR = {1998},
    NUMBER = {2-3},
     PAGES = {87--96},
      ISSN = {0024-3795,1873-1856},
   MRCLASS = {15A09},
  MRNUMBER = {1645022},
MRREVIEWER = {Ar.\ Meenakshi},
       DOI = {10.1016/S0024-3795(98)00008-1},
       URL = {https://doi.org/10.1016/S0024-3795(98)00008-1},
}

@article {MR1937251,
    AUTHOR = {Wei, Yimin and Wu, Hebing},
     TITLE = {The representation and approximation for the generalized
              inverse {$A^{(2)}_{T,S}$}},
   JOURNAL = {Appl. Math. Comput.},
  FJOURNAL = {Applied Mathematics and Computation},
    VOLUME = {135},
      YEAR = {2003},
    NUMBER = {2-3},
     PAGES = {263--276},
      ISSN = {0096-3003,1873-5649},
   MRCLASS = {15A09},
  MRNUMBER = {1937251},
MRREVIEWER = {Predrag\ S.\ Stanimirovi\'c},
       DOI = {10.1016/S0096-3003(01)00327-7},
       URL = {https://doi.org/10.1016/S0096-3003(01)00327-7},
}

@article {MR227186,
    AUTHOR = {Urquhart, N. S.},
     TITLE = {Computation of generalized inverse matrices which satisfy
              specified conditions},
   JOURNAL = {SIAM Rev.},
  FJOURNAL = {SIAM Review. A Publication of the Society for Industrial and
              Applied Mathematics},
    VOLUME = {10},
      YEAR = {1968},
     PAGES = {216--218},
      ISSN = {1095-7200},
   MRCLASS = {15.15},
  MRNUMBER = {227186},
MRREVIEWER = {M.\ Lotkin},
       DOI = {10.1137/1010035},
       URL = {https://doi.org/10.1137/1010035},
}

@book {MR1987382,
    AUTHOR = {Ben-Israel, Adi and Greville, Thomas N. E.},
     TITLE = {Generalized inverses},
    SERIES = {CMS Books in Mathematics/Ouvrages de Math\'ematiques de la
              SMC},
    VOLUME = {15},
   EDITION = {Second},
      NOTE = {Theory and applications},
 PUBLISHER = {Springer-Verlag, New York},
      YEAR = {2003},
     PAGES = {xvi+420},
      ISBN = {0-387-00293-6},
   MRCLASS = {15A09 (47A05 47A50)},
  MRNUMBER = {1987382},
}

@book {MR3793648,
    AUTHOR = {Wang, Guorong and Wei, Yimin and Qiao, Sanzheng},
     TITLE = {Generalized inverses: theory and computations},
    SERIES = {Developments in Mathematics},
    VOLUME = {53},
   EDITION = {Second},
 PUBLISHER = {Springer, Singapore; Science Press Beijing, Beijing},
      YEAR = {2018},
     PAGES = {xix+378},
      ISBN = {978-981-13-0145-2; 978-981-13-0146-9},
   MRCLASS = {15-02 (15A09 47A05 65Fxx)},
  MRNUMBER = {3793648},
MRREVIEWER = {Maria\ Celeste\ Gouveia},
       DOI = {10.1007/978-981-13-0146-9},
       URL = {https://doi.org/10.1007/978-981-13-0146-9},
}

@article {MR69793,
    AUTHOR = {Penrose, R.},
     TITLE = {A generalized inverse for matrices},
   JOURNAL = {Proc. Cambridge Philos. Soc.},
  FJOURNAL = {Proceedings of the Cambridge Philosophical Society},
    VOLUME = {51},
      YEAR = {1955},
     PAGES = {406--413},
      ISSN = {0008-1981},
   MRCLASS = {09.0X},
  MRNUMBER = {69793},
MRREVIEWER = {O.\ Taussky-Todd},
}

@book{wei2018quaternion,
  title={Quaternion matrix computations},
  author={Wei, Musheng and Li, Ying and Zhang, Fengxia and Zhao, Jianli},
  year={2018},
  publisher={Nova Science Publishers, Hauppauge, NY}
}

@book {Strogatz2015,
    AUTHOR = {Strogatz, Steven H.},
     TITLE = {Nonlinear dynamics and chaos},
   EDITION = {Second},
      NOTE = {With applications to physics, biology, chemistry, and
              engineering},
 PUBLISHER = {Westview Press, Boulder, CO},
      YEAR = {2015},
     PAGES = {xiii+513},
      ISBN = {978-0-8133-4910-7; 978-0-8133-4911-4},
   MRCLASS = {37-01 (28A80 34-01 34C28 37Cxx 37D45)},
  MRNUMBER = {3837141},
MRREVIEWER = {Anca\ R\v adulescu and Richard\ P.\ Halpern},
}

@article {MR2802523,
    AUTHOR = {Kyrchei, I. I.},
     TITLE = {Determinantal representations of the {M}oore-{P}enrose inverse
              over the quaternion skew field and corresponding {C}ramer's
              rules},
   JOURNAL = {Linear Multilinear Algebra},
  FJOURNAL = {Linear and Multilinear Algebra},
    VOLUME = {59},
      YEAR = {2011},
    NUMBER = {4},
     PAGES = {413--431},
      ISSN = {0308-1087,1563-5139},
   MRCLASS = {15A15 (15A09 15B33)},
  MRNUMBER = {2802523},
MRREVIEWER = {Xiaoji\ Liu},
       DOI = {10.1080/03081081003586860},
       URL = {https://doi.org/10.1080/03081081003586860},
}

@article {MR3322229,
    AUTHOR = {Xu, Yi and Yu, Licheng and Xu, Hongteng and Zhang, Hao and
              Nguyen, Truong},
     TITLE = {Vector sparse representation of color image using quaternion
              matrix analysis},
   JOURNAL = {IEEE Trans. Image Process.},
  FJOURNAL = {IEEE Transactions on Image Processing},
    VOLUME = {24},
      YEAR = {2015},
    NUMBER = {4},
     PAGES = {1315--1329},
      ISSN = {1057-7149,1941-0042},
   MRCLASS = {94A08},
  MRNUMBER = {3322229},
       DOI = {10.1109/TIP.2015.2397314},
       URL = {https://doi.org/10.1109/TIP.2015.2397314},
}

@article {MR4471046,
    AUTHOR = {Chen, Junren and Ng, Michael K.},
     TITLE = {Color image inpainting via robust pure quaternion matrix
              completion: error bound and weighted loss},
   JOURNAL = {SIAM J. Imaging Sci.},
  FJOURNAL = {SIAM Journal on Imaging Sciences},
    VOLUME = {15},
      YEAR = {2022},
    NUMBER = {3},
     PAGES = {1469--1498},
      ISSN = {1936-4954},
   MRCLASS = {65D18 (65F10 94A08 97N30)},
  MRNUMBER = {4471046},
       DOI = {10.1137/22M1476897},
       URL = {https://doi.org/10.1137/22M1476897},
}

@article {MR4947673,
    AUTHOR = {Deng, Zhanwang and Su, Yuqiu and Huang, Wen},
     TITLE = {S{LRQA}: a sparse low-rank quaternion model for color image
              processing with convergence analysis},
   JOURNAL = {J. Sci. Comput.},
  FJOURNAL = {Journal of Scientific Computing},
    VOLUME = {105},
      YEAR = {2025},
    NUMBER = {1},
     PAGES = {Paper No. 5, 44},
      ISSN = {0885-7474,1573-7691},
   MRCLASS = {90C26 (17A35 65K05 68U10)},
  MRNUMBER = {4947673},
       DOI = {10.1007/s10915-025-03019-4},
       URL = {https://doi.org/10.1007/s10915-025-03019-4},
}

@article {MR4039278,
    AUTHOR = {Chen, Yongyong and Xiao, Xiaolin and Zhou, Yicong},
     TITLE = {Low-rank quaternion approximation for color image processing},
   JOURNAL = {IEEE Trans. Image Process.},
  FJOURNAL = {IEEE Transactions on Image Processing},
    VOLUME = {29},
      YEAR = {2020},
     PAGES = {1426--1439},
      ISSN = {1057-7149,1941-0042},
   MRCLASS = {94A08 (90C26 90C90)},
  MRNUMBER = {4039278},
       DOI = {10.1109/TIP.2019.2941319},
       URL = {https://doi.org/10.1109/TIP.2019.2941319},
}

@article {MR3979957,
    AUTHOR = {Jia, Zhigang and Ng, Michael K. and Song, Guang-Jing},
     TITLE = {Robust quaternion matrix completion with applications to image
              inpainting},
   JOURNAL = {Numer. Linear Algebra Appl.},
  FJOURNAL = {Numerical Linear Algebra with Applications},
    VOLUME = {26},
      YEAR = {2019},
    NUMBER = {4},
     PAGES = {e2245, 35},
      ISSN = {1070-5325,1099-1506},
   MRCLASS = {65F99 (15A83 65K05 90C25 94A08)},
  MRNUMBER = {3979957},
       DOI = {10.1002/nla.2245},
       URL = {https://doi.org/10.1002/nla.2245},
}

\end{document}